\documentclass[preprint]{elsarticle}

\usepackage{graphicx}%
\usepackage{multirow}%
\usepackage{amsmath,amssymb,amsfonts}%
\usepackage{amsthm}%
\usepackage{mathrsfs}%
\usepackage[title]{appendix}%
\usepackage{xcolor}%
\usepackage{textcomp}%
\usepackage{manyfoot}%
\usepackage{booktabs}%
\usepackage{float}
\usepackage{placeins}%
\usepackage{algorithm}%
\usepackage{algorithmicx}%
\usepackage{algpseudocode}%
\usepackage{listings}%
\usepackage{hyperref} 
\usepackage{subcaption}
\usepackage{array}
\usepackage{geometry} % 用于设置页面边距

\newcommand{\R}{\mathbb{R}}
\newcommand{\B}{\mathbb{B}}
\newcommand{\eps}{\epsilon}
\newcommand{\ip}[2]{\langle #1,#2\rangle}
\newcommand{\norm}[1]{\|#1\|}

\newcommand{\prox}{\operatorname{prox}}
\newcommand{\argmin}{\operatorname*{argmin}}
\newcommand{\argmax}{\operatorname*{argmax}}
\newcommand{\cS}{\mathcal{S}}
\newcommand{\cL}{\mathcal L}
\newcommand{\cF}{\mathcal F}

\usepackage{tabularx}
\usepackage{color}
\usepackage{multirow}
\usepackage{mathtools}
\usepackage{enumerate}
\usepackage{makecell}
\newenvironment{herefigure}{\par\smallskip\noindent\begin{minipage}{\linewidth}\centering}{\end{minipage}\par\smallskip}

\newcommand{\msc}[1]{\textbf{MSC Code}#1}
\theoremstyle{thmstyleone}%
\newtheorem{theorem}{Theorem}[section]%  meant for continuous numbers
\newtheorem{proposition}{Proposition}[section]%

\newtheorem{lemma}{Lemma}[section]

\theoremstyle{thmstyletwo}%
\newtheorem{example}{Example}[section]%
\newtheorem{remark}{Remark}[section]%

\theoremstyle{thmstylethree}%
\newtheorem{definition}{Definition}[section]%
\newtheorem{assumption}{Assumption}[section]
\journal{Communications in Nonlinear Science and Numerical Simulation}

\begin{document}

\begin{frontmatter}



\title{A switching proximal gradient method for constrained difference-of-convex optimization problems}


%% Authors
\author[1]{Yuncheng Liu}
\ead{yuncheng.liu@swun.edu.cn}

\author[2]{Chuangliang Zhang}
\ead{clzhang1992@sina.com}

\author[3]{Bo Yang\corref{cor1}}
\ead{bbo_yang@163.com}

\cortext[cor1]{Corresponding author.}

%% Affiliations
\affiliation[1]{
    organization={School of Mathematics, Southwest Minzu University},
    %addressline={},
    city={Chengdu},
    postcode={610041},
    state={Sichuan},
    country={P.R. China}
}

\affiliation[2]{
    organization={School of Mathematics, Jiaying University},
    %addressline={},
    city={Meizhou},
    postcode={514015},
    state={Guangdong},
    country={P.R. China}
}

\affiliation[3]{
    organization={School of Mathematics, Chengdu Normal University},
    %addressline={},
    city={Chengdu},
    postcode={611130},
    state={Sichuan},
    country={P.R. China}
}
%% Abstract
\begin{abstract}
We study difference-of-convex optimization problems with several composite
inequality constraints and develop a majorization-guided switching proximal
gradient (MGSPG) method. Specifically, at each feasible outer point, quadratic majorization and linearization yield a strongly convex composite surrogate problem. To solve this surrogate approximately, we employ a switching proximal gradient (SPG) method, which takes a proximal gradient step for the objective when the current point is nearly feasible and for a most violated constraint otherwise. This construction requires neither projection onto the joint surrogate feasible set nor multiplier computation. By establishing a common invariant ball and using an increasing weight argument, we obtain an \(O(\tau^{-1})\) inner complexity bound for SPG. Building on this bound, we show that, under boundedness of the feasible level set and a uniform Slater condition for the surrogates, MGSPG finds an \(\varepsilon\)-KKT point in \(O(\varepsilon^{-2})\) outer iterations and \(O(\varepsilon^{-4})\) selected component proximal gradient updates. Numerical experiments on SCAD-constrained and groupwise-constrained regression problems demonstrate the viability and computational efficiency of our method compared with existing optimization algorithms.
\end{abstract}

%% Keywords

\end{frontmatter}


%Version 3.1 December 2024
% See section 11 of the User Manual for version history
%
%%%%%%%%%%%%%%%%%%%%%%%%%%%%%%%%%%%%%%%%%%%%%%%%%%%%%%%%%%%%%%%%%%%%%%
%%                                                                 %%
%% Please do not use \input{...} to include other tex files.       %%
%% Submit your LaTeX manuscript as one .tex document.              %%
%%                                                                 %%
%% All additional figures and files should be attached             %%
%% separately and not embedded in the \TeX\ document itself.       %%
%%                                                                 %%
%%%%%%%%%%%%%%%%%%%%%%%%%%%%%%%%%%%%%%%%%%%%%%%%%%%%%%%%%%%%%%%%%%%%%
\msc{: 65Y20, 90C06, 90C26, 90C55}
\section{Introduction}%\label{sec:intro}
 In this paper, we consider
the following composite inequality constrained difference-of-convex (DC) optimization problems 
\begin{equation}\label{eq:problem}
\begin{aligned}
  \min_{x\in\R^n}\quad
  & F(x):=f(x)+P_1(x)-P_2(x)\\
  \mathrm{s.t.}\quad
  & G_i(x):=g_i(x)+Q_{i1}(x)-Q_{i2}(x)\le0,
  \qquad i\in[m]:=\{1,\ldots,m\},
\end{aligned}
\end{equation}
where $f$, $g_i$, $P_1,P_2,Q_{i1}$ and $Q_{i2}$  satisfy Assumption \ref{ass:data} below.
\begin{assumption}\label{ass:data}
In problem~\eqref{eq:problem},
\begin{enumerate}[(i)]
\item the functions $f:\R^n\to\R$ and $g_i:\R^n\to\R$, $i\in[m]$,
are continuously differentiable, but need not be convex. Their
gradients are Lipschitz continuous with constants $L_f$ and
$L_{g_i}$, respectively;
\item the functions $P_1,P_2,Q_{i1},Q_{i2}:\R^n\to\R$ are
convex;
\item there exists an $x^0 \in\R^n$ such that  $G_i(x^0)\le0$, $i\in[m]$ and the set
\begin{equation*}%\label{eq:omega0}
  \Omega_0:=  \left\{ x\in\R^n:  G_i(x)\le0,\ i\in[m],\ F(x)\le F(x^0) \right\}
\end{equation*}
is bounded.
\end{enumerate}
\end{assumption}
\noindent Assumption~\ref{ass:data}(i)--(ii) implies that $F$ and every
$G_i$ are continuous, and hence that $\Omega_0$ is closed.  Since
$\Omega_0$ is bounded by~(iii), it is compact.
Continuity of $F$ therefore gives
\begin{equation*}%\label{eq:inf}
  F_{\inf} := \inf\{F(x):G_i(x)\le0,\ i\in[m]\}  =  \min_{x\in\Omega_0}F(x) >-\infty.
\end{equation*}
For computation, we assume access to function values, the gradients
of $f$ and $g_i$, a subgradient of each $P_2$ and $Q_{i2}$. For $\gamma>0$, the
proximal mappings $\prox_{\gamma P_1}$ and
$\prox_{\gamma Q_{i1}}$ can be computed efficiently.
For background on convex subdifferentials and proximal mappings, see
\cite{Rockafellar1970,Beck2017}. The separation of smooth and proximable
terms is central to forward--backward splitting, accelerated
proximal-gradient methods, and broader proximal algorithms
\cite{CombettesWajs2005,BeckTeboulle2009,ParikhBoyd2014}.

Functional constraints express separate requirements on groups or
design scenarios~\cite{KomiyamaTakedaHondaShimao2018,
CotterEtAl2019Constraints,BenTalNemirovski1997}, while DC terms allow
nonconvex regularization~\cite{LeThiPhamDinh2018}.
DC decompositions also underlie DCA and proximal DC algorithms for
structured nonconvex objectives
\cite{LeThiPhamDinh2005,WenChenPong2018}.
The retained terms $P_1$ and $Q_{i1}$ may have simple individual
proximal mappings even when their joint constrained problem is
difficult to solve. The following models illustrate these applications.

\begin{example}[Robust sparse regression with groupwise validation
constraints]\label{ex:application-regression}
Motivated by groupwise prediction requirements~\cite{AgarwalDudikWu2019},
we combine a Lorentzian loss~\cite{CarrilloBarnerAysal2010}, the
minimax concave penalty (MCP)~\cite{Zhang2010}, and weighted
$\ell_1-\ell_2$ penalties~\cite{YinLouHeXin2015}:
\begin{equation*}\label{eq:application-regression}
\begin{aligned}
\min_{x\in\R^n}\quad
&\frac1{N_0}\sum_{r=1}^{N_0}\ell_{\kappa_0}(a_r^\top x-b_r)
 +\frac{\beta}{2}\norm{x}^2
 +\sum_{j=1}^n p_{\lambda_{\rm reg},a_{\rm mcp}}(x_j)\\
\mathrm{s.t.}\quad
&\frac1{N_i}\sum_{r=1}^{N_i}\ell_{\kappa_i}(c_{ir}^\top x-d_{ir})
 -\nu_i+\alpha_i\bigl(\norm{D_i x}_1-\norm{D_i x}_2\bigr)\le0,
 \quad i\in[m].
\end{aligned}
\end{equation*}
Here $a_r,c_{ir}\in\R^n$ and $b_r,d_{ir}\in\R$ are training and
groupwise validation data, $N_i\ge1$ are integers for $i=0,\ldots,m$,
$D_i$ is positive diagonal, and $\nu_i$ is a prescribed threshold. The loss is
$\ell_\kappa(t)=\frac{\kappa^2}{2}\log(1+t^2/\kappa^2)$,
and $p_{\lambda_{\rm reg},a_{\rm mcp}}$ denotes the MCP with
$\lambda_{\rm reg}>0$ and $a_{\rm mcp}>1$.
We take $\beta>0$, $\kappa_i>0$ for $i=0,\ldots,m$, and
$\alpha_i>0$ for $i\in[m]$.
\end{example}

\begin{example}[Sparse design with smoothed spectral constraints]
\label{ex:application-spectral}
Motivated by semidefinite models in robust design
\cite{BenTalNemirovski1997}, we use log-trace-exp smoothing
\cite{Nesterov2007} and the preceding penalties to form
\begin{equation*}\label{eq:application-spectral}
\begin{aligned}
\min_{x\in\R^n}\quad
&\frac12\norm{Bx-c}^2+\frac{\beta}{2}\norm{x}^2
 +\sum_{j=1}^n p_{\lambda_{\rm reg},a_{\rm mcp}}(x_j)\\
\mathrm{s.t.}\quad
&\phi_{\delta_i}(\mathcal A_i(x))-\nu_i
 +\alpha_i\bigl(\norm{D_i x}_1-\norm{D_i x}_2\bigr)\le0,
 \quad i\in[m],
\end{aligned}
\end{equation*}
where $B\in\R^{q\times n}$, $c\in\R^q$, and
\[
\mathcal A_i(x):=A_{i0}+\sum_{j=1}^n x_jA_{ij},
\qquad
\phi_{\delta_i}(Y):=
\delta_i\log\operatorname{tr}\exp(Y/\delta_i).
\]
The matrices $A_{ij}\in\mathbb S^{r_i}$ are symmetric,
$\delta_i>0$, and $\nu_i$ is the prescribed spectral threshold.
The penalty parameters and matrices $D_i$ are as in
Example~\ref{ex:application-regression}.
\end{example}

Existing first-order methods often handle functional constraints by replacing the original problem
with a sequence of convex subproblems. For example, the moving-balls method
uses quadratic upper models to generate feasible iterates for
smooth nonlinear inequalities~\cite{AuslenderShefiTeboulle2010}.
Related work studies majorization--minimization
procedures~\cite{BoltePauwels2016}, sequential convex programming
with line search for constrained DC
problems~\cite{YuPongLu2021}, and smoothing methods for DC
objectives with convex composite or nonlinear conic
constraints~\cite{XuPongZhang2026,XuPongSze2025}.
Feasible sequential and proximal regularization methods solve
constrained local subproblems, possibly
inexactly~\cite{FacchineiLamparielloScutari2017,BoobDengLan2023}.
For convex problems, saddle-point subgradient methods provide
primal--dual guarantees~\cite{NedicOzdaglar2009}, stochastic
primal--dual methods use sampled constraint
queries~\cite{Xu2020}, and multiproximal methods combine
component models in a joint update~\cite{BolteChenPauwels2020}.
These approaches make the solution of convex subproblems
an essential part of the algorithm. One representative work that accommodates distinct convex
nonsmooth terms in the objective and constraints is the
level-constrained proximal-gradient method
of~\cite{BoobDengLan2025}. Specifically, its computational
assumption requires a joint quadratically constrained composite
problem to be solved exactly or to a prescribed accuracy.
However, access to the individual proximal mappings does not
by itself provide this solver. When these mappings are the
available building blocks, an inner method must specify how
to use them and quantify the work needed to attain the
required accuracy.

This issue arises naturally in the majorization construction
for~\eqref{eq:problem}. At a feasible outer iterate, linearizing
the subtracted convex terms and using strongly convex quadratic
upper models for the smooth terms gives
\begin{equation}\label{majorizationproblem}
\begin{aligned}
 \min_{z\in\R^n}\quad &H_0(z):=h_0(z)+R_0(z)\\
 \mathrm{s.t.}\quad &H_i(z):=h_i(z)+R_i(z)\le0,
 \qquad i\in[m],
\end{aligned}
\end{equation}
where each $h_i$ is smooth and strongly convex and each $R_i$
is finite-valued, convex, and proximable. Since these models
majorize the original functions and agree with them at the
outer iterate, the surrogate is feasible and its feasible
points satisfy the original constraints. An approximate inner
solution, however, requires bounds on both constraint violation
and objective error before it can be used in the outer scheme.

Switching methods provide a natural starting point for the inner
problem. A basic scheme takes an objective step at a sufficiently
feasible point and a constraint step otherwise~\cite{Polyak1967}.
Related constructions include CoMirror
\cite{BeckEtAl2010CoMirror}, cooperative stochastic approximation
\cite{LanZhou2020}, and adaptive mirror descent
\cite{BayandinaEtAl2018}. This principle has also been analyzed for
weakly convex constraints~\cite{MaLinYang2020,HuangLin2023},
strongly convex proximal subproblems~\cite{JiaGrimmer2025},
and convexified DC subproblems~\cite{HeHuangYaoLin2025}. For composite problems, the treatment of the nonsmooth terms
determines the required computational operations.
Grimmer and Li~\cite{GrimmerLi2025} use a proximal mapping for
the objective's nonsmooth term and subgradients for constraint
corrections. Nesterov~\cite[Sections~5 and 6]{Nesterov2024}
studies max-linearized models with a common composite term and
Bregman projections onto individual linearized constraints.
In our setting, the available building blocks are the individual
pairs $(\nabla h_i,\prox_{\gamma R_i})$; these do not generally
provide the proximal mapping of the maximum constraint function
through one component call. A direct precedent is Titov et al.\
\cite[Remark~2.4 and Appendix~A]{TitovEtAl2021}, whose switching
mirror steps retain each selected component's own convex term.
With Euclidean geometry and smooth linearized parts, these steps
reduce to proximal gradient updates. Their relatively Lipschitz
convex analysis gives $O(\tau^{-2})$ complexity for objective and
feasibility accuracy $\tau$.

The strongly convex surrogates considered here motivate a sharper
inner analysis within this individual proximal gradient framework.
We seek an $O(\tau^{-1})$ complexity bound for the one-sided
objective residual and constraint violation, with constants
uniform over feasible outer centers under the stated assumptions.
A common invariant ball controls the switching trajectory without
global component Lipschitz bounds or a prescribed bounded domain,
while strong convexity and weighted estimates yield the required
accuracy guarantee. Building on these composite switching updates, we analyze the
following switching proximal gradient (SPG) scheme for the
strongly convex majorization surrogate \eqref{majorizationproblem}. Specifically, at an inner iterate
$z^t$, choose
\[
 i_t\in\operatorname*{argmax}_{i\in[m]}H_i(z^t),
 \qquad
 V_t:=H_{i_t}(z^t).
\]
For a tolerance $\tau>0$ and stepsize $\gamma_t>0$, set
\[
 z^{t+1}
 =
 \begin{cases}
 \displaystyle
 \prox_{\gamma_t R_0}
 \bigl(z^t-\gamma_t\nabla h_0(z^t)\bigr),
 & V_t\le\tau,\\[2mm]
 \displaystyle
 \prox_{\gamma_t R_{i_t}}
 \bigl(z^t-\gamma_t\nabla h_{i_t}(z^t)\bigr),
 & V_t>\tau.
 \end{cases}
\]
Thus each update uses one component's gradient and proximal
mapping, retaining its own nonsmooth term. The maximum constraint
value determines the mode and selected constraint; its proximal
mapping is not needed. 

Embedding SPG in feasible DC majorization gives the
majorization-guided switching proximal-gradient method (MGSPG).
The method requires neither projection onto the joint surrogate feasible set nor
multiplier computation. Our main contributions are threefold.
First, we establish a common invariant ball for the objective
and constraint proximal gradient updates of each surrogate,
controlling the switching trajectory without global component
Lipschitz bounds or a prescribed bounded domain.
Second, strong convexity and weighted component estimates ensure
that the objective-step weighted average attains maximum
constraint violation and one-sided objective residual at most
$\tau$ within $O(\tau^{-1})$ selected component updates.
Under the stated assumptions, the constants are uniform over
feasible outer centers.
Third, under a uniform surrogate Slater condition for a fixed
joint subgradient selection, uniform multiplier bounds yield
a calibrated outer displacement test. With the prescribed inner
accuracy, MGSPG either certifies the current center as an
$\eps$-KKT point in the sense of Section~\ref{sec:model}
or takes a feasible descent step. This yields
$O(\eps^{-2})$ outer iterations and $O(\eps^{-4})$ selected
component proximal gradient updates, with constants independent
of the target accuracy. Exact surrogate solutions and
multipliers enter only the analysis.

The remainder of the paper is organized as follows. Section~\ref{sec:model} introduces
the notation, stationarity notions, surrogate models, and
regularity condition. Section~\ref{sec:inner} develops SPG and
MGSPG and proves the inner accuracy bound. Using this bound,
Section~\ref{sec:outer} establishes the outer stationarity and
complexity guarantees. Section~\ref{sec:numerical} presents numerical
experiments. Finally, we give some concluding remarks in Section \ref{Conclusion}.

% Notation and preliminaries: replacement section; filename follows the requested name.
\section{Notation and preliminaries}
\label{sec:model}

In this paper, we use \(\mathbb{R}^n\) to denote the \(n\)-dimensional Euclidean space, and we use \(\langle \cdot, \cdot \rangle\) and \(\|\cdot\|\) to denote its inner product and the associated norm, respectively. 
For $x\in\R^n$ and $r\ge0$, we write
$\B(x,r):=\{z\in\R^n:\norm{z-x}\le r\}$, and for
$a\in\R$, we set $[a]_+:=\max\{a,0\}$.
Given a nonempty closed set $C\subseteq\R^n$, its distance
function is $\operatorname{dist}(x,C):=\inf_{y\in C}\norm{x-y}$,
and its indicator function is
\[
\delta_C(x):=
\begin{cases}
0, & x\in C,\\
+\infty, & x\notin C.
\end{cases}
\]
If $C$ is also convex, its normal cone at $x\in C$ is
\[
N_C(x):=\{v\in\R^n:\ip{v}{y-x}\le0 \quad \forall y\in C\}.
\]

For an extended-real-valued function
$h:\R^n\to[-\infty,+\infty]$, define
$\operatorname{dom}h:=\{x\in\R^n:h(x)<+\infty\}$.
We call $h$ proper if it never takes the value $-\infty$
and $\operatorname{dom}h\ne\emptyset$, and closed if it is
lower semicontinuous. For a proper closed convex function $h$,
the subdifferential at $x\in\operatorname{dom}h$ is
\[
\partial h(x):=
\{\xi\in\R^n:h(u)-h(x)\ge\ip{\xi}{u-x}
 \quad \forall u\in\R^n\}.
\]
For $\gamma>0$, its proximal map is
\[
\prox_{\gamma h}(v)
:=\argmin_{x\in\R^n}
\Big\{h(x)+\frac1{2\gamma}\norm{x-v}^2\Big\}.
\]
It is single-valued for every $v\in\R^n$; see, for example,
\cite[Chapter~1]{Beck2017}.
We denote the feasible set of \eqref{eq:problem} by $\cF:=\{x\in\R^n:G_i(x)\le0,\ i\in[m]\}$.

The model construction and its analysis use both sides of the
smoothness inequality. Specifically, if $h:\R^n\to\R$ is
differentiable with an $L_h$-Lipschitz gradient, then
\begin{equation}\label{eq:two-sided-smoothness}
-\frac{L_h}{2}\norm{y-x}^2
\le h(y)-h(x)-\ip{\nabla h(x)}{y-x}
\le\frac{L_h}{2}\norm{y-x}^2,
\qquad x,y\in\R^n.
\end{equation}
The upper estimate supplies the quadratic majorants used below.
The lower estimate implies that $h+L_h\norm{\cdot}^2/2$ is
convex, that is, $h$ is $L_h$-weakly convex. More generally,
$h$ is $\omega$-weakly convex, for $\omega\ge0$, if
$h+\omega\norm{\cdot}^2/2$ is convex.

We next specify the stationarity notions used in the convergence
analysis. Since a DC decomposition is generally nonunique,
both definitions below refer to the fixed decomposition in
\eqref{eq:problem}.

\begin{definition}[KKT point]
\label{def:kkt}
A point $x\in\R^n$ is a KKT point of \eqref{eq:problem} if
there exist $\lambda\in\R_+^m$ and
\[
p_1\in\partial P_1(x),\quad p_2\in\partial P_2(x),\qquad
q_{i1}\in\partial Q_{i1}(x),\quad q_{i2}\in\partial Q_{i2}(x)
\]
such that
\[
\begin{gathered}
0=\nabla f(x)+p_1-p_2+
\sum_{i=1}^m\lambda_i
\bigl(\nabla g_i(x)+q_{i1}-q_{i2}\bigr),\\
G_i(x)\le0,\qquad
\lambda_iG_i(x)=0,\quad i\in[m].
\end{gathered}
\]
\end{definition}

For approximate stationarity, we allow the retained and
linearized convex terms to be evaluated at two nearby points.
Motivated by \cite[Definition~6.2]{HeHuangYaoLin2025}, the
following variant requires exact feasibility at the nearby
point and bounds the sum of the complementarity residuals.

\begin{definition}[(Approximate) KKT point]
\label{def:near-kkt}
Let $\eps_1,\eps_2\ge0$. A point $x\in\R^n$ is an
$(\eps_1,\eps_2)$-KKT point of \eqref{eq:problem} if there
exist $z\in\R^n$, $\lambda\in\R_+^m$, and
\[
p_1\in\partial P_1(z),\quad p_2\in\partial P_2(x),\qquad
q_{i1}\in\partial Q_{i1}(z),\quad q_{i2}\in\partial Q_{i2}(x)
\]
such that
\[
\begin{gathered}
\Big\|
\nabla f(z)+p_1-p_2+
\sum_{i=1}^m\lambda_i
\bigl(\nabla g_i(z)+q_{i1}-q_{i2}\bigr)
\Big\|\le\eps_1,\\
G_i(z)\le0\quad(i\in[m]),\qquad
\sum_{i=1}^m|\lambda_iG_i(z)|\le\eps_1,\qquad
\norm{z-x}\le\eps_2.
\end{gathered}
\]
When $\eps_1=\eps_2=\eps$, we simply say that $x$ is an
$\eps$-KKT point.
For $\eps=0$, we have $z=x$ and every term
$|\lambda_iG_i(x)|$ vanishes. Thus this definition reduces
to Definition~\ref{def:kkt}.
\end{definition}

This two-point evaluation is consistent with the surrogate
construction: the subtracted convex terms are linearized at
the current center, while $P_1$ and $Q_{i1}$ are retained.
For each $i\in[m]$, fix a selection
$\zeta_i(x)\in\partial Q_{i2}(x)$ on $\Omega_0$; such
selections exist because $Q_{i2}$ is finite convex.
The resulting joint selection is fixed throughout, and the
uniform Slater condition below applies to the models it generates.

Choose $\rho_f,\rho_{g_i}>0$ and set
$\mu_f:=L_f+\rho_f$ and $\mu_{g_i}:=L_{g_i}+\rho_{g_i}$.
At a center $x^k\in\Omega_0$, choose
$\xi^k\in\partial P_2(x^k)$, set
$\zeta_i^k:=\zeta_i(x^k)$, and define
\begin{align*}
\Phi_k(z)
&:=f(x^k)+\ip{\nabla f(x^k)-\xi^k}{z-x^k}
+\frac{\mu_f}{2}\norm{z-x^k}^2
+P_1(z)-P_2(x^k),\\
\Psi_{i,k}(z)
&:=g_i(x^k)+\ip{\nabla g_i(x^k)-\zeta_i^k}{z-x^k}
+\frac{\mu_{g_i}}2\norm{z-x^k}^2
+Q_{i1}(z)-Q_{i2}(x^k).
\end{align*}
By convexity of $P_1$ and $Q_{i1}$, these models are
$\mu_f$- and $\mu_{g_i}$-strongly convex, respectively.
Moreover, the upper estimate in \eqref{eq:two-sided-smoothness}
and the subgradient inequalities for $P_2$ and $Q_{i2}$ give
\[
\begin{aligned}
F(z)+\frac{\rho_f}{2}\norm{z-x^k}^2
&\le\Phi_k(z),\\
G_i(z)+\frac{\rho_{g_i}}2\norm{z-x^k}^2
&\le\Psi_{i,k}(z),\qquad i\in[m],
\end{aligned}
\]
for every $z\in\R^n$. At the center, the models agree with
the original functions:
\begin{equation}\label{eq:center-equality}
\Phi_k(x^k)=F(x^k),
\qquad
\Psi_{i,k}(x^k)=G_i(x^k),\quad i\in[m].
\end{equation}
These properties lead to the surrogate problem
\begin{equation}\label{eq:surrogate}
\begin{aligned}
\min_{z\in\R^n}\quad
&\Phi_k(z)\\
\mathrm{s.t.}\quad
&\Psi_{i,k}(z)\le0,\quad i\in[m],
\end{aligned}
\end{equation}
with feasible set
$\cS_k:=\{z\in\R^n:\Psi_{i,k}(z)\le0,\ i\in[m]\}$.
Since $x^k$ is feasible, \eqref{eq:center-equality} gives
$x^k\in\cS_k$. In addition, majorization implies
$\cS_k\subseteq\cF$.

To analyze the surrogate KKT systems uniformly, we need a
quantitative form of strict feasibility. For a fixed surrogate,
Slater's condition guarantees KKT multipliers at an optimal
solution \cite[Corollary~28.2.1 and Theorem~28.3]{Rockafellar1970}.
However, the constraints vary with the center and the selected
subgradients, so separate Slater conditions do not by themselves
give a uniform multiplier bound. We therefore require a common
strict-feasibility margin, as in related analyses of regularized
and convexified subproblems
\cite{MaLinYang2020,HeHuangYaoLin2025,BoobDengLan2023,BoobDengLan2025}.
Together with the uniform bounds established below, this margin
will control the surrogate multipliers.

To state the condition independently of the iteration index,
for $x\in\Omega_0$ and $\zeta_i\in\partial Q_{i2}(x)$, define
\[
\Psi_{i,x,\zeta_i}(z)
:=g_i(x)+\ip{\nabla g_i(x)-\zeta_i}{z-x}
+\frac{\mu_{g_i}}2\norm{z-x}^2
+Q_{i1}(z)-Q_{i2}(x),
\]
and write $\Psi_{i,x}:=\Psi_{i,x,\zeta_i(x)}$ for the model
generated by the prescribed selection.

\begin{assumption}[Uniform surrogate Slater condition]
\label{ass:slater}
There exists $\sigma_S>0$ such that, for every $x\in\Omega_0$,
some $s_x\in\R^n$ satisfies
\[
\Psi_{i,x}(s_x)\le-\sigma_S,
\qquad i\in[m].
\]
\end{assumption}

\begin{remark}%\label{rem:uniform-surrogate-slater}
The witness $s_x$ may depend on the center, but it must satisfy
all surrogate constraints with the same margin $\sigma_S$.
The assumption concerns the prescribed joint selection and
requires no continuity of that selection. Notice that
$\Psi_{i,x}(x)=G_i(x)\le0$ already guarantees feasibility;
the uniform strict margin is used to bound the multipliers.
\end{remark}

\begin{remark}[A sufficient margin for the regression model]
\label{rem:regression-slater}
We conclude by verifying the uniform margin for
Example~\ref{ex:application-spectral}.
Write its group loss as $\mathcal L_i$, so that
$g_i=\mathcal L_i-\nu_i$. The lower smoothness inequality gives
\[
g_i(x)-\ip{\nabla g_i(x)}{x}
\le \mathcal L_i(0)-\nu_i+\frac{L_{g_i}}2\norm{x}^2.
\]
Positive homogeneity of $Q_{i2}$ and $Q_{i1}(0)=0$ then yield
\[
\Psi_{i,x,\zeta_i}(0)
\le \mathcal L_i(0)-\nu_i
+\frac{2L_{g_i}+\rho_{g_i}}2\norm{x}^2.
\]
Since $F(x)\ge\beta\norm{x}^2/2$, a sufficient calibration is
\[
\nu_i\ge \mathcal L_i(0)
+\frac{(2L_{g_i}+\rho_{g_i})F(0)}{\beta}+\bar\delta,
\qquad\bar\delta>0.
\]
It gives $s_x=0$ and $\sigma_S=\bar\delta$ uniformly over the level set
and all admissible selections. This is a sufficient condition, not
a necessary one. The smaller thresholds used in
Section~\ref{sec:numerical} establish strict feasibility of the
initial point but do not verify this condition. Nor has a uniform
surrogate margin been established for the reported SCAD instances.
Accordingly, the theoretical guarantees remain conditional on
Assumption~\ref{ass:slater}; the numerical examples do not constitute
a verification of that assumption.
\end{remark}

\section{Algorithm}
\label{sec:inner}

We present a two-level method based on successive surrogate
construction. At each outer iteration, we form
\eqref{eq:surrogate} and apply the switching proximal gradient
method developed below. The inner method uses constraint values
to select either an objective step or a step for a most violated
constraint. Each update requires the gradient and proximal map
of the selected component, together with the values needed for
switching, and avoids projection onto the joint surrogate feasible
set. We first analyze the inner method for a generic strongly
convex composite problem and then specialize it to the DC surrogate.

\subsection{Switching proximal gradient method}
%\label{subsec:spg}

Consider the strongly convex composite problem
\begin{equation}\label{gccp}
\begin{aligned}
\min_{z\in\R^n}\quad
& H_0(z):=h_0(z)+R_0(z)\\
\mathrm{s.t.}\quad
& H_i(z):=h_i(z)+R_i(z)\le0,\quad i\in[m].
\end{aligned}
\end{equation}
Write $[m]_0:=\{0,1,\ldots,m\}$. We impose the following
assumptions on the component functions and the feasible set.

\begin{assumption}\label{ass:inner}
For every $i\in[m]_0$, the function $h_i$ is continuously
differentiable, $\mu_i$-strongly convex, and has an
$L_i$-Lipschitz gradient, where $0<\mu_i\le L_i$.
The function $R_i:\R^n\to\R$ is convex, and its value and
proximal map are available. The set
\[
\mathcal C:=\{z\in\R^n:H_i(z)\le0,\ i\in[m]\}
\]
is nonempty.
\end{assumption}

For the surrogate \eqref{eq:surrogate}, the required strong
convexity follows from its affine--quadratic construction;
the original smooth functions $f$ and $g_i$ need not be strongly
convex. This property will be used to establish a common
invariant ball and, together with the prescribed stepsizes
and weights, the $O(T^{-1})$ inner bound.

Under Assumption~\ref{ass:inner}, each $H_i$ is finite,
continuous, and strongly convex. Hence $\mathcal C$ is closed
and convex. Moreover, the strong subgradient inequality gives
a coercive quadratic lower bound for $H_0$. Since $\mathcal C$
is nonempty, $H_0$ attains its minimum over $\mathcal C$, and
strong convexity makes the minimizer unique. We denote it
by $z^*$.

We first record the proximal gradient inequality used to
analyze both the objective and constraint updates; see
\cite[Chapter~10]{Beck2017}.

\begin{lemma}\label{lem:strong-pg}
Let $h$ be $\mu$-strongly convex with an $L_h$-Lipschitz
gradient, let $R$ be finite convex, and let
\[
y=\prox_{\gamma R}(x-\gamma\nabla h(x)),
\qquad 0<\gamma\le L_h^{-1}.
\]
Then, for all $z\in\R^n$,
\begin{equation}\label{eq:strong-pg}
h(y)+R(y)
\le
h(z)+R(z)
+\frac{1-\gamma\mu}{2\gamma}\norm{x-z}^2
-\frac{1}{2\gamma}\norm{y-z}^2.
\end{equation}
\end{lemma}

\begin{proof}
The proximal optimality condition gives
$r_y:=-\nabla h(x)-\gamma^{-1}(y-x)\in\partial R(y)$.
Combining the descent lemma, the subgradient inequality for
$R$, and strong convexity of $h$, we obtain
\[
\begin{aligned}
h(y)+R(y)
&\le h(x)+\ip{\nabla h(x)}{y-x}
+\frac{L_h}{2}\norm{y-x}^2
+R(z)-\ip{r_y}{z-y}\\
&\le h(z)+R(z)-\frac{\mu}{2}\norm{z-x}^2
+\frac1\gamma\ip{y-x}{z-y}
+\frac{L_h}{2}\norm{y-x}^2\\
&=h(z)+R(z)
+\frac{1-\gamma\mu}{2\gamma}\norm{x-z}^2
-\frac{1}{2\gamma}\norm{y-z}^2\\
&\quad
-\frac{1-\gamma L_h}{2\gamma}\norm{y-x}^2.
\end{aligned}
\]
Here the equality uses
$2\ip{y-x}{z-y}
=\norm{x-z}^2-\norm{y-z}^2-\norm{y-x}^2$.
Since $\gamma L_h\le1$, dropping the last nonpositive term
proves \eqref{eq:strong-pg}.
\end{proof}

The preceding estimate applies to each component separately.
To control the switching sequence, we next introduce bounds
that are uniform over the selected component. For
$i\in[m]_0$ and $\gamma>0$, define
\[
T_{i,\gamma}(z)
:=\prox_{\gamma R_i}
\bigl(z-\gamma\nabla h_i(z)\bigr).
\]
Fix an initial point $z^0$ and choose
$r_i^0\in\partial R_i(z^0)$ for each $i\in[m]_0$;
such subgradients exist because each $R_i$ is finite convex.
Set
\[
\mu_{\min}:=\min_{i\in[m]_0}\mu_i,
\qquad
B_0:=\max_{i\in[m]_0}
\norm{\nabla h_i(z^0)+r_i^0},
\qquad
\gamma_{\min}:=\min_{i\in[m]_0}\frac{\mu_i}{L_i^2}.
\]
These constants determine a common radius and stepsize bound.
Specifically, choose $R>0$ satisfying
\begin{equation}\label{eq:invariant-radius}
R\ge\frac{2B_0}{\mu_{\min}},
\end{equation}
and, on the closed ball $\B(z^0,R)$, define the subgradient bound
\[
M_R:=
\max_{i\in[m]_0}
\sup_{\substack{z\in\B(z^0,R)\\r\in\partial R_i(z)}}
\norm{\nabla h_i(z)+r}.
\]
The constant $M_R$ is finite because the subdifferentials
of finite convex functions and the continuous gradients are
bounded on compact sets. We will show that the ball is invariant
under every component map when $\gamma\le\gamma_{\min}$;
thus $M_R$ will apply along the entire switching sequence.

We now specify the stepsizes and averaging weights.
Choose an integer $t_0\ge0$ sufficiently large and set
\begin{equation}\label{eq:decreasing-step}
\gamma_t:=\frac{2}{\mu_{\min}(t+t_0+2)},
\quad
w_t:=t+t_0+1,
\quad
\gamma_0=\frac{2}{\mu_{\min}(t_0+2)}
\le\gamma_{\min}.
\end{equation}
With these choices, Algorithm~\ref{alg:spg} takes an objective
step when all constraints are within tolerance $\tau$ and
otherwise takes a step for a most violated constraint.

\begin{algorithm}[H]
\caption{The switching proximal gradient method
\(\operatorname{SPG}(\tau,T,z^0;t_0,\{\gamma_t\})\)}
\label{alg:spg}
\begin{algorithmic}[1]
\Require $\tau>0$, an integer $T\ge1$, $z^0\in\R^n$,
and an integer $t_0\ge0$ with stepsizes
$\{\gamma_t\}_{t=0}^{T-1}$ and weights
$\{w_t\}_{t=0}^{T-1}$ given by~\eqref{eq:decreasing-step}.
\State Set $I\gets\emptyset$ and $J\gets\emptyset$.
\For{$t=0,1,\ldots,T-1$}
 \State Choose
 $\displaystyle i_t\in\argmax_{i\in[m]}H_i(z^t)$
 and set $V_t\gets H_{i_t}(z^t)$.
 \If{$V_t\le\tau$}
 \State
 $\displaystyle
 z^{t+1}\gets
 \prox_{\gamma_tR_0}
 \bigl(z^t-\gamma_t\nabla h_0(z^t)\bigr)$,
 \quad $I\gets I\cup\{t\}$.
 \Else
 \State
 $\displaystyle
 z^{t+1}\gets
 \prox_{\gamma_tR_{i_t}}
 \bigl(z^t-\gamma_t\nabla h_{i_t}(z^t)\bigr)$,
 \quad $J\gets J\cup\{t\}$.
 \EndIf
\EndFor
\If{$I=\emptyset$}
 \State \textbf{Return} failure
 (insufficient inner iteration budget).
\EndIf
\State \textbf{Return}
$\bar z_T:=
\frac{\sum_{t\in I}w_tz^t}{\sum_{t\in I}w_t}$.
\end{algorithmic}
\end{algorithm}

The sets $I$ and $J$ record objective and constraint steps,
respectively. In particular, every point $z^t$ with $t\in I$
satisfies $H_i(z^t)\le\tau$ for all $i\in[m]$, which motivates
averaging these points in the return step. The failure branch
accounts for budgets for which $I$ is empty. To derive a budget
that guarantees $I\ne\emptyset$, we first establish boundedness
of the iterates.

An update for one component may increase another, so individual
sublevel sets need not be preserved by switching. The following
lemma gives a ball invariant under every component update.

\begin{lemma}\label{lem:common-invariance}
If $0<\gamma\le\gamma_{\min}$, then
\[
T_{i,\gamma}\bigl(\B(z^0,R)\bigr)
\subseteq\B(z^0,R)
\quad\forall\,i\in[m]_0.
\]
\end{lemma}

\begin{proof}
Fix $i\in[m]_0$ and $0<\gamma\le\gamma_{\min}$.
We first bound the contraction of $T_{i,\gamma}$.
Indeed, nonexpansiveness of $\prox_{\gamma R_i}$,
strong monotonicity, and Lipschitz continuity of
$\nabla h_i$ give
\begin{align}
\norm{T_{i,\gamma}(z)-T_{i,\gamma}(y)}^2
&\le
\norm{z-y-\gamma(\nabla h_i(z)-\nabla h_i(y))}^2
\nonumber\\
&\le
\bigl(1-2\gamma\mu_i+\gamma^2L_i^2\bigr)
\norm{z-y}^2
\nonumber\\
&\le
(1-\gamma\mu_i)\norm{z-y}^2.
\label{eq:component-contraction}
\end{align}
The last inequality uses $\gamma\le\mu_i/L_i^2$.
We next control the displacement of the center $z^0$.
Set $u:=T_{i,\gamma}(z^0)$. The proximal optimality condition
gives
$r_u:=-\nabla h_i(z^0)-\gamma^{-1}(u-z^0)\in\partial R_i(u)$.
Hence, by monotonicity of $\partial R_i$,
\[
0
\le
\ip{r_u-r_i^0}{u-z^0}
=
-\ip{\nabla h_i(z^0)+r_i^0}{u-z^0}
-\frac1\gamma\norm{u-z^0}^2.
\]
Cauchy--Schwarz therefore yields
\begin{equation}\label{eq:component}
\norm{T_{i,\gamma}(z^0)-z^0}
\le
\gamma\norm{\nabla h_i(z^0)+r_i^0}
\le\gamma B_0.
\end{equation}
Finally, $\gamma\mu_i\le\mu_i^2/L_i^2\le1$ implies
$\sqrt{1-\gamma\mu_i}\le1-\gamma\mu_i/2$.
Combining \eqref{eq:component-contraction}--\eqref{eq:component}
with \eqref{eq:invariant-radius}, we obtain, for every
$z\in\B(z^0,R)$,
\begin{align*}
\norm{T_{i,\gamma}(z)-z^0}
&\le
\norm{T_{i,\gamma}(z)-T_{i,\gamma}(z^0)}
+\norm{T_{i,\gamma}(z^0)-z^0}\\
&\le
\left(1-\frac{\gamma\mu_i}{2}\right)R+\gamma B_0\\
&\le R.
\end{align*}
Hence $T_{i,\gamma}(z)\in\B(z^0,R)$.
\end{proof}

Lemma~\ref{lem:common-invariance} allows us to use the same
bound $M_R$ at every iteration. We now combine this bound
with Lemma~\ref{lem:strong-pg} to estimate the selected
component at the current point $z^t$, where the switching
test is performed.

\begin{lemma}
\label{lem:inner-contraction}
Suppose that Assumption~\ref{ass:inner} holds, that $R$ satisfies
\eqref{eq:invariant-radius}, and that the stepsizes are given by
\eqref{eq:decreasing-step}. Define
\[
\ell_t:=
\begin{cases}
0, & t\in I,\\
i_t, & t\in J.
\end{cases}
\]
Then $z^t\in\B(z^0,R)$ for $t=0,\ldots,T$.
Moreover, for $t=0,\ldots,T-1$,
\begin{equation}\label{eq:current-contraction}
H_{\ell_t}(z^t)-H_{\ell_t}(z^*)
\le
\frac{1-\gamma_t\mu_{\min}}{2\gamma_t}
\norm{z^t-z^*}^2
-\frac1{2\gamma_t}\norm{z^{t+1}-z^*}^2
+M_R^2\gamma_t.
\end{equation}
\end{lemma}

\begin{proof}
Since $\gamma_t\le\gamma_0\le\gamma_{\min}$,
Lemma~\ref{lem:common-invariance} and induction give
$z^t\in\B(z^0,R)$ for $t=0,\ldots,T$.

Fix $t<T$. To transfer the proximal gradient estimate from
$z^{t+1}$ to $z^t$, we first bound the step length.
The proximal optimality condition gives
\[
r^{t+1}:=
-\nabla h_{\ell_t}(z^t)
-\gamma_t^{-1}(z^{t+1}-z^t)
\in\partial R_{\ell_t}(z^{t+1}).
\]
For any $r^t\in\partial R_{\ell_t}(z^t)$,
monotonicity of $\partial R_{\ell_t}$ therefore implies
\[
\gamma_t^{-1}\norm{z^{t+1}-z^t}^2
\le
-\ip{\nabla h_{\ell_t}(z^t)+r^t}{z^{t+1}-z^t}.
\]
Since $z^t\in\B(z^0,R)$, Cauchy--Schwarz and the definition
of $M_R$ yield
\begin{equation}\label{eq:uniform-inner-step}
\norm{z^{t+1}-z^t}\le M_R\gamma_t.
\end{equation}
Moreover,
$\nabla h_{\ell_t}(z^t)+r^t\in\partial H_{\ell_t}(z^t)$.
Thus the subgradient inequality and
\eqref{eq:uniform-inner-step} give
\begin{equation}\label{eq:selected-component-change}
\begin{aligned}
H_{\ell_t}(z^t)-H_{\ell_t}(z^{t+1})
&\le
\ip{\nabla h_{\ell_t}(z^t)+r^t}{z^t-z^{t+1}}\\
&\le
M_R\norm{z^{t+1}-z^t}
\le M_R^2\gamma_t.
\end{aligned}
\end{equation}
On the other hand,
$\gamma_t\le\gamma_{\min}\le
\mu_{\ell_t}/L_{\ell_t}^2\le1/L_{\ell_t}$.
Applying Lemma~\ref{lem:strong-pg} to the selected component
with comparison point $z^*$, we obtain
\[
H_{\ell_t}(z^{t+1})-H_{\ell_t}(z^*)
\le
\frac{1-\gamma_t\mu_{\ell_t}}{2\gamma_t}
\norm{z^t-z^*}^2
-\frac1{2\gamma_t}\norm{z^{t+1}-z^*}^2.
\]
Adding \eqref{eq:selected-component-change} and using
$\mu_{\ell_t}\ge\mu_{\min}$ proves
\eqref{eq:current-contraction}.
\end{proof}

We now derive the iteration complexity of SPG for \eqref{gccp}.
The reciprocal stepsizes and linearly increasing weights follow the
weighted-telescoping approach of
\cite{LacosteJulienSchmidtBach2012}. Here the additional ingredients
are the common invariant ball and the separation of objective and
constraint updates. Multiplying \eqref{eq:current-contraction} by $w_t$ makes
the distance terms telescope. The resulting error bound grows
at most linearly in $T$, whereas the total weight grows
quadratically. The switching rule then guarantees an objective
step within the prescribed budget and yields the following
certificate.

\begin{theorem}[Complexity of SPG]
\label{thm:inner}
Suppose that $\tau>0$ and that the assumptions of
Lemma~\ref{lem:inner-contraction} hold, and define
\[
A_0:=
\frac{\mu_{\min}}4t_0(t_0+1)\norm{z^0-z^*}^2.
\]
If
\begin{equation}\label{eq:inner-budget}
T\ge
\left\lceil
\max\left\{
1,\,
2\sqrt{\frac{A_0}{\tau}},\,
\frac{8M_R^2}{\mu_{\min}\tau}
\right\}
\right\rceil,
\end{equation}
then $I\ne\emptyset$, and the output satisfies
\begin{equation}\label{eq:inner-certificate}
\max_{i\in[m]}H_i(\bar z_T)\le\tau,
\qquad
H_0(\bar z_T)-H_0(z^*)\le\tau.
\end{equation}
Consequently, choosing $T$ equal to the ceiling in
\eqref{eq:inner-budget} gives $O(\tau^{-1})$ selected component
proximal gradient updates as $\tau\downarrow0$, with the
problem data, $z^0$, $R$, and $t_0$ fixed. If a full scan
identifies a most violated constraint, the number of
constraint-value evaluations is $O(m\tau^{-1})$.
\end{theorem}

\begin{proof}
We first derive a weighted bound over all iterations.
Substituting the stepsize in \eqref{eq:decreasing-step}
into \eqref{eq:current-contraction} gives
\begin{align*}
H_{\ell_t}(z^t)-H_{\ell_t}(z^*)
&\le
\frac{\mu_{\min}}4
(t+t_0)\norm{z^t-z^*}^2
-\frac{\mu_{\min}}4
(t+t_0+2)\norm{z^{t+1}-z^*}^2
\nonumber\\
&\quad+
\frac{2M_R^2}{\mu_{\min}(t+t_0+2)}.
%\label{eq:unweighted-recursion}
\end{align*}
Multiplying by $w_t=t+t_0+1$, we obtain
\begin{align*}
w_t\bigl(H_{\ell_t}(z^t)-H_{\ell_t}(z^*)\bigr)
\le{}&
\frac{\mu_{\min}}4
(t+t_0)(t+t_0+1)\norm{z^t-z^*}^2\\
&-\frac{\mu_{\min}}4
(t+t_0+1)(t+t_0+2)\norm{z^{t+1}-z^*}^2\\
&+\frac{2M_R^2}{\mu_{\min}}
\frac{t+t_0+1}{t+t_0+2}.
\end{align*}
The distance coefficients agree at consecutive iterations.
Summing over $t=0,\ldots,T-1$ therefore yields
\begin{align}
\sum_{t=0}^{T-1}w_t
\bigl(H_{\ell_t}(z^t)-H_{\ell_t}(z^*)\bigr)
&\le
A_0-\frac{\mu_{\min}}4
(T+t_0)(T+t_0+1)\norm{z^T-z^*}^2
\nonumber\\
&\quad+
\frac{2M_R^2}{\mu_{\min}}
\sum_{t=0}^{T-1}\frac{t+t_0+1}{t+t_0+2}
\nonumber\\
&\le
A_0+\frac{2M_R^2}{\mu_{\min}}T,
\label{eq:global-weighted-bound}
\end{align}
where the last inequality drops the nonpositive terminal
term and uses $(t+t_0+1)/(t+t_0+2)\le1$. Moreover, the
total weight satisfies
\begin{equation*}\label{eq:total-weight}
W_T:=\sum_{t=0}^{T-1}w_t
=\frac{T(T+2t_0+1)}2
\ge\frac{T^2}2.
\end{equation*}
Consequently, \eqref{eq:global-weighted-bound} gives
\begin{equation*}\label{eq:normalized-inner-bound}
\frac1{W_T}
\sum_{t=0}^{T-1}w_t
\bigl(H_{\ell_t}(z^t)-H_{\ell_t}(z^*)\bigr)
\le
\frac{2A_0}{T^2}
+\frac{4M_R^2}{\mu_{\min}T}.
\end{equation*}
The budget \eqref{eq:inner-budget} bounds each term on the
right-hand side by $\tau/2$. Thus
\begin{equation}\label{eq:weighted-budget}
\sum_{t=0}^{T-1}w_t
\bigl(H_{\ell_t}(z^t)-H_{\ell_t}(z^*)\bigr)
\le\tau W_T.
\end{equation}

We next show that $I\ne\emptyset$. For every $t\in J$,
the switching rule gives $H_{i_t}(z^t)>\tau$, while
feasibility of $z^*$ gives $H_{i_t}(z^*)\le0$. Hence
\begin{equation}\label{eq:constraint-step-gap}
H_{i_t}(z^t)-H_{i_t}(z^*)>\tau.
\end{equation}
If $I=\emptyset$, this strict inequality holds at every
iteration, and summation gives
\[
\sum_{t=0}^{T-1}w_t
\bigl(H_{\ell_t}(z^t)-H_{\ell_t}(z^*)\bigr)
>\tau W_T,
\]
contradicting \eqref{eq:weighted-budget}. Therefore
$I\ne\emptyset$, and the returned average is well defined.

It remains to bound its objective value and constraint
violation. Let
\[
W_I:=\sum_{t\in I}w_t,
\qquad
W_J:=\sum_{t\in J}w_t.
\]
Since $W_T=W_I+W_J$, subtracting the constraint-step
contribution in \eqref{eq:weighted-budget} and using
\eqref{eq:constraint-step-gap} yield
\[
\sum_{t\in I}w_t\bigl(H_0(z^t)-H_0(z^*)\bigr)
\le\tau W_T-\tau W_J=\tau W_I.
\]
Convexity of $H_0$ now implies
\[
H_0(\bar z_T)-H_0(z^*)
\le
\frac1{W_I}\sum_{t\in I}w_t
\bigl(H_0(z^t)-H_0(z^*)\bigr)
\le\tau.
\]
Similarly, $H_i(z^t)\le\tau$ for every $t\in I$ and
$i\in[m]$. By convexity of each $H_i$,
\[
H_i(\bar z_T)
\le\frac1{W_I}\sum_{t\in I}w_tH_i(z^t)
\le\tau,
\qquad i\in[m].
\]
This proves \eqref{eq:inner-certificate}.

Finally, with the problem data, $z^0$, $R$, and $t_0$ fixed,
the two accuracy-dependent terms in \eqref{eq:inner-budget}
are $O(\tau^{-1/2})$ and $O(\tau^{-1})$. The prescribed
budget is therefore $O(\tau^{-1})$ as $\tau\downarrow0$.
Each iteration performs one selected component
proximal gradient update and, under a full scan, evaluates
$m$ constraint values. The stated complexity bounds follow.
\end{proof}

Theorem~\ref{thm:inner} provides the objective and feasibility
certificate required by the outer majorization method.
To turn its theoretical budget into a uniform choice for
the DC surrogates, we still need to control
$\norm{z^0-z^*}$ in $A_0$ and the local bound $M_R$.
We derive bounds independent of the outer iteration next
and use them to choose a common inner budget.

% MGSPG subsection: clean replacement text for the original manuscript.
\subsection{Majorization-Guided Switching Proximal-Gradient Method}
%\label{subsec:mgspg}

We now apply SPG to the DC surrogate \eqref{eq:surrogate}.
For a center $x^k\in\Omega_0$, \eqref{eq:center-equality}
gives $x^k\in\cS_k$. Moreover, the smooth components
\[
h_0^k:=\Phi_k-P_1,
\qquad
h_i^k:=\Psi_{i,k}-Q_{i1},\quad i\in[m],
\]
have affine gradients, with both smoothness and strong-convexity
moduli equal to $\mu_f$ for $h_0^k$ and $\mu_{g_i}$ for $h_i^k$.
Together with the retained convex terms $P_1$ and $Q_{i1}$,
the assumed component oracles, and feasibility of $x^k$, these
properties verify Assumption~\ref{ass:inner}. Set
\[
\hat\mu_{\min}:=
\min\{\mu_f,\mu_{g_1},\ldots,\mu_{g_m}\},
\qquad
\hat\mu_{\max}:=
\max\{\mu_f,\mu_{g_1},\ldots,\mu_{g_m}\}.
\]
Thus the generic inner constants satisfy
\[
\mu_{\min}=\hat\mu_{\min},
\qquad
\max_{i\in[m]_0}L_i=\hat\mu_{\max},
\qquad
\gamma_{\min}=\frac1{\hat\mu_{\max}}.
\]

We now  majorization guided
switching proximal gradient method (MGSPG). At each iteration,
SPG produces a candidate point. If its displacement from the
current center is small, MGSPG returns the current center;
otherwise, feasibility and descent safeguards determine whether
the candidate becomes the next center. We retain these safeguards
for arbitrary inner tolerances and show below that the calibration
\eqref{eq:mgspg-calibration} guarantees acceptance whenever
the stopping test fails.

\begin{algorithm}[H]
\caption{Majorization Guided Switching Proximal Gradient Method (MGSPG)}
\label{alg:mgspg}
\begin{algorithmic}[1]
\Require
Initial feasible point $x^0$, inner tolerance $\tau>0$,
integer $T\ge1$, stopping threshold $\eta>0$, and integer
$t_0\ge0$ such that
$\gamma_0=2/[\hat\mu_{\min}(t_0+2)]\le1/\hat\mu_{\max}$,
with $\gamma_t:=2/[\hat\mu_{\min}(t+t_0+2)]$.
\For{$k=0,1,\ldots$}
    \State Choose $\xi^k\in\partial P_2(x^k)$ and set
    $\zeta_i^k=\zeta_i(x^k)$ for all $i\in[m]$ using the fixed
    joint selection mapping preceding Assumption~\ref{ass:slater}.
    \State Construct surrogates $\Phi_k,\Psi_{1,k},\ldots,\Psi_{m,k}$.
    \State Set
    $x^{k+1}:=\operatorname{SPG}(\tau,T,x^k;t_0,\{\gamma_t\})$
    applied to~\eqref{eq:surrogate}.
    \If{$\norm{x^{k+1}-x^k}\le\eta$}
        \State \textbf{Terminate and return} $x^k$.
    \EndIf
    \If{$\max_{i\in[m]}G_i(x^{k+1})>0$ \textbf{or}
    $F(x^{k+1})>F(x^k)$}
        \State \textbf{Terminate and report}
        ``insufficient inner accuracy''.
    \EndIf
\EndFor
\end{algorithmic}
\end{algorithm}

At a feasible center, the inner call is well defined for every
integer $T\ge1$. Indeed, \eqref{eq:center-equality} gives
$\max_i\Psi_{i,k}(x^k)\le0<\tau$, so the first SPG iteration
is an objective step. Hence $0\in I$, and the returned average
is defined. The safeguards then either terminate MGSPG or
accept the candidate as the next center. In particular, they
preserve the level set used throughout the analysis.

\begin{remark}%\label{rem:level-set-invariance}
Every outer center belongs to $\Omega_0$. To see this, suppose
that $x^k\in\Omega_0$ and the algorithm continues. The safeguards
give $G_i(x^{k+1})\le0$ for all $i\in[m]$ and
$F(x^{k+1})\le F(x^k)\le F(x^0)$, so $x^{k+1}\in\Omega_0$.
Since $x^0\in\Omega_0$, induction proves the claim. A rejected
candidate is not used as a new center, and the small-displacement
branch returns the current feasible center $x^k$.
\end{remark}

We next derive bounds uniform over the outer centers and the
admissible linearizations. The argument has three steps:
we first localize the surrogate feasible sets and their exact
solutions, then bound the associated KKT multipliers, and
finally choose a common SPG budget. These bounds will support
the outer alternative between returning an approximate KKT
point and taking a feasible descent step.

For the localization argument, compactness of $\Omega_0$,
continuity of the gradients, and local boundedness of the
subdifferentials give a constant $M_{\rm lev}\ge1$ such that
\[
\begin{aligned}
\norm{\nabla f(x)}+\norm{p_1}+\norm{p_2}
&\le M_{\rm lev},\\
\norm{\nabla g_i(x)}+\norm{q_{i1}}+\norm{q_{i2}}
&\le M_{\rm lev},\qquad i\in[m],
\end{aligned}
\]
for all $x,y_1,y_2\in\Omega_0$,
$p_\ell\in\partial P_\ell(y_\ell)$, and
$q_{i\ell}\in\partial Q_{i\ell}(y_\ell)$,
$\ell\in\{1,2\}$. Set
\[
M_{\rm slack}:=
\max_{\substack{x\in\Omega_0\\i\in[m]}}[-G_i(x)],
\qquad
R_S:=
\frac{M_{\rm lev}+
\sqrt{M_{\rm lev}^2+2\hat\mu_{\min}M_{\rm slack}}}
{\hat\mu_{\min}}.
\]
Both constants are finite and independent of the outer index.
Moreover, feasibility on $\Omega_0$ gives $M_{\rm slack}\ge0$,
and hence $R_S\ge2M_{\rm lev}/\hat\mu_{\min}>0$.
The next lemma explains the role of this radius.

\begin{lemma}[Localization and exact surrogate descent]
\label{lem:surrogate-localization}
The constant $R_S$ defined above is independent of $k$, and,
for every $x^k\in\Omega_0$,
\begin{equation}\label{eq:S-localization}
\cS_k\subseteq\B(x^k,R_S),
\qquad
\operatorname{diam}(\cS_k)\le2R_S.
\end{equation}
Moreover, problem~\eqref{eq:surrogate} has a unique solution
$\hat{x}^{k+1}\in\Omega_0$, and
\begin{equation}\label{eq:strong-surrogate-descent}
\Phi_k(\hat{x}^{k+1})
\le
F(x^k)-\frac{\mu_f}{2}
\norm{\hat{x}^{k+1}-x^k}^2.
\end{equation}
\end{lemma}

\begin{proof}
Fix $x^k\in\Omega_0$. We first prove \eqref{eq:S-localization}.
Let $v\in\cS_k$, fix $i\in[m]$, and choose
$\hat q_{i1}^k\in\partial Q_{i1}(x^k)$.
The subgradient inequality for $Q_{i1}$ and the definition of
$\Psi_{i,k}$ give
\begin{align*}
\Psi_{i,k}(v)
&\ge
G_i(x^k)
+\ip{\nabla g_i(x^k)+\hat q_{i1}^k-\zeta_i^k}{v-x^k}
+\frac{\mu_{g_i}}2\norm{v-x^k}^2\\
&\ge
-M_{\rm slack}-M_{\rm lev}\norm{v-x^k}
+\frac{\hat\mu_{\min}}2\norm{v-x^k}^2,
\end{align*}
where the second inequality uses Cauchy--Schwarz, the definitions
of $M_{\rm slack}$ and $M_{\rm lev}$, and
$\mu_{g_i}\ge\hat\mu_{\min}$. Since $\Psi_{i,k}(v)\le0$,
$r:=\norm{v-x^k}$ satisfies
\[
\frac{\hat\mu_{\min}}2r^2-M_{\rm lev}r-M_{\rm slack}\le0.
\]
The positive root of this quadratic is $R_S$, while its other
root is nonpositive. Thus $r\le R_S$, which proves the inclusion
in \eqref{eq:S-localization}; the diameter bound follows by
the triangle inequality.

We next consider the exact surrogate solution. By
\eqref{eq:center-equality}, $x^k\in\cS_k$, so $\cS_k$ is
nonempty. Each $\Psi_{i,k}$ is finite, continuous, and convex;
hence $\cS_k$ is closed and convex. Together with
\eqref{eq:S-localization}, this shows that $\cS_k$ is compact.
Since $\Phi_k$ is continuous and $\mu_f$-strongly convex,
it has a unique minimizer over $\cS_k$, denoted by
$\hat{x}^{k+1}$.

Moreover, optimality of $\hat{x}^{k+1}$, the center equalities,
and the majorization inequalities yield
\begin{align*}
F(\hat{x}^{k+1})
+\frac{\rho_f}{2}\norm{\hat{x}^{k+1}-x^k}^2
&\le\Phi_k(\hat{x}^{k+1})
\le\Phi_k(x^k)=F(x^k),\\
G_i(\hat{x}^{k+1})
+\frac{\rho_{g_i}}2\norm{\hat{x}^{k+1}-x^k}^2
&\le\Psi_{i,k}(\hat{x}^{k+1})\le0,
\qquad i\in[m].
\end{align*}
Consequently, $\hat{x}^{k+1}$ is feasible for the original
problem and satisfies $F(\hat{x}^{k+1})\le F(x^k)\le F(x^0)$.
Thus $\hat{x}^{k+1}\in\Omega_0$.

Finally, we prove \eqref{eq:strong-surrogate-descent}.
Since $\Phi_k$ is finite and continuous, the exact convex sum
rule and constrained optimality give
\[
0\in\partial\Phi_k(\hat{x}^{k+1})
+N_{\cS_k}(\hat{x}^{k+1}).
\]
Choose $a_k\in\partial\Phi_k(\hat{x}^{k+1})$ such that
$-a_k\in N_{\cS_k}(\hat{x}^{k+1})$. Because $x^k\in\cS_k$,
the normal-cone inequality gives
$\ip{a_k}{x^k-\hat{x}^{k+1}}\ge0$.
Strong convexity of $\Phi_k$ therefore implies
\begin{align*}
\Phi_k(x^k)
&\ge
\Phi_k(\hat{x}^{k+1})
+\ip{a_k}{x^k-\hat{x}^{k+1}}
+\frac{\mu_f}{2}\norm{\hat{x}^{k+1}-x^k}^2\\
&\ge
\Phi_k(\hat{x}^{k+1})
+\frac{\mu_f}{2}\norm{\hat{x}^{k+1}-x^k}^2.
\end{align*}
Using $\Phi_k(x^k)=F(x^k)$ proves
\eqref{eq:strong-surrogate-descent}.
\end{proof}

The exact surrogate solution serves as an analytical reference
for the computed SPG point. To use its KKT system uniformly
over the outer iterations, we next bound its multipliers.
The localization just established and the uniform Slater margin
together provide the required bound.

\begin{proposition}[Uniform surrogate multiplier bound]
\label{prop:multiplier-bound}
Suppose that Assumptions~\ref{ass:data} and \ref{ass:slater}
hold. For every $x^k\in\Omega_0$, the surrogate problem admits
a KKT multiplier. Every such multiplier $\lambda_k\in\R_+^m$
can be paired with
$p_1^k\in\partial P_1(\hat{x}^{k+1})$ and
$q_{i1}^k\in\partial Q_{i1}(\hat{x}^{k+1})$, $i\in[m]$,
so that the stationarity, feasibility, and complementarity conditions
for \eqref{eq:surrogate} hold.
Moreover, there is a constant $\Lambda\ge0$, independent of
$k$, such that every such multiplier satisfies
\begin{equation}\label{eq:Mlambda}
\|\lambda_k\|_1\le\Lambda.
\end{equation}
\end{proposition}

\begin{proof}
The uniform surrogate Slater condition and the convex KKT theorem
\cite[Corollary~28.2.1 and Theorem~28.3]{Rockafellar1970}
ensure that the multiplier set is nonempty. Fix any KKT
multiplier $\lambda_k\in\R_+^m$. By the exact subdifferential
sum rule \cite[Theorem~23.8]{Rockafellar1970}, we can choose
\[
p_1^k\in\partial P_1(\hat{x}^{k+1}),
\qquad
q_{i1}^k\in\partial Q_{i1}(\hat{x}^{k+1}),\quad i\in[m],
\]
such that
\begin{equation}\label{eq:surrogate-KKT-expanded}
\begin{aligned}
0={}&\nabla f(x^k)+\mu_f(\hat{x}^{k+1}-x^k)+p_1^k-\xi^k\\
&+\sum_{i=1}^m\lambda_{i,k}
\bigl(\nabla g_i(x^k)+\mu_{g_i}(\hat{x}^{k+1}-x^k)
+q_{i1}^k-\zeta_i^k\bigr),\\
\Psi_{i,k}(\hat{x}^{k+1})&\le0,\qquad
\lambda_{i,k}\Psi_{i,k}(\hat{x}^{k+1})=0,
\quad i\in[m].
\end{aligned}
\end{equation}
For brevity, write
\[
\begin{aligned}
u_k&:=\nabla f(x^k)+\mu_f(\hat{x}^{k+1}-x^k)+p_1^k-\xi^k,\\
w_{i,k}&:=\nabla g_i(x^k)+\mu_{g_i}(\hat{x}^{k+1}-x^k)
+q_{i1}^k-\zeta_i^k.
\end{aligned}
\]
Then $u_k\in\partial\Phi_k(\hat{x}^{k+1})$,
$w_{i,k}\in\partial\Psi_{i,k}(\hat{x}^{k+1})$, and
$u_k+\sum_i\lambda_{i,k}w_{i,k}=0$.
Since $x^k,\hat{x}^{k+1}\in\Omega_0$ and
$\norm{\hat{x}^{k+1}-x^k}\le R_S$, the definition of
$M_{\rm lev}$ gives
\[
\norm{u_k}\le M_{\rm lev}+\mu_fR_S.
\]

Let $s^k:=s_{x^k}$ be the Slater point from
Assumption~\ref{ass:slater}. The subgradient inequality,
the bound $\Psi_{i,k}(s^k)\le-\sigma_S$, and complementarity
yield, for every $i\in[m]$,
\[
\begin{aligned}
\lambda_{i,k}\ip{w_{i,k}}{\hat{x}^{k+1}-s^k}
&\ge\lambda_{i,k}
\bigl(\Psi_{i,k}(\hat{x}^{k+1})-\Psi_{i,k}(s^k)\bigr)\\
&\ge\sigma_S\lambda_{i,k}.
\end{aligned}
\]
Summing these inequalities and using stationarity, we obtain
\[
\begin{aligned}
\sigma_S\|\lambda_k\|_1
&\le\sum_{i=1}^m\lambda_{i,k}
\ip{w_{i,k}}{\hat{x}^{k+1}-s^k}\\
&=-\ip{u_k}{\hat{x}^{k+1}-s^k}
\le\norm{u_k}\,\norm{\hat{x}^{k+1}-s^k}.
\end{aligned}
\]
Both $s^k$ and $\hat{x}^{k+1}$ belong to $\cS_k$.
Hence \eqref{eq:S-localization} gives
$\norm{\hat{x}^{k+1}-s^k}\le2R_S$, and therefore
\[
\|\lambda_k\|_1
\le\frac{2R_S(M_{\rm lev}+\mu_fR_S)}{\sigma_S}
=:\Lambda.
\]
This bound is independent of $k$ and of the chosen KKT
multiplier, proving \eqref{eq:Mlambda}.
\end{proof}

We now turn from the exact surrogate solution to its SPG
approximation. To apply Theorem~\ref{thm:inner} with a common
budget, it remains to bound the inner constants $A_0$ and
$M_R$ uniformly and to choose a single admissible shift $t_0$.
The preceding localization provides the required starting point.

\begin{proposition}[Uniform SPG oracle]
\label{prop:uniform-spg}
Suppose that Assumption~\ref{ass:data} holds. There exist
$C_{\rm spg}>0$ and an integer $t_0\ge0$, independent of $k$
and of the admissible choice of $\xi^k\in\partial P_2(x^k)$,
with the following property. For every $x^k\in\Omega_0$ and
$0<\tau\le1$, set
\begin{equation}\label{eq:Ttau}
T_\tau:=\left\lceil\frac{C_{\rm spg}}\tau\right\rceil.
\end{equation}
Then Algorithm~\ref{alg:spg}, started from $x^k$ with this
$t_0$ and run for $T_\tau$ iterations, returns a point
$x^{k+1}:=\bar z_{T_\tau}^k$ satisfying
\begin{equation}\label{eq:uniform-oracle}
\max_{i\in[m]}\Psi_{i,k}(x^{k+1})\le\tau,
\qquad
\Phi_k(x^{k+1})-\Phi_k(\hat{x}^{k+1})\le\tau.
\end{equation}
\end{proposition}

\begin{proof}
We first choose a shift valid for every surrogate. The smooth
components introduced above have gradients
\[
\begin{aligned}
\nabla h_0^k(v)
&=\nabla f(x^k)-\xi^k+\mu_f(v-x^k),\\
\nabla h_i^k(v)
&=\nabla g_i(x^k)-\zeta_i^k+\mu_{g_i}(v-x^k),
\qquad i\in[m].
\end{aligned}
\]
Thus $\mu_{\min}=\hat\mu_{\min}$ and
$\gamma_{\min}=1/\hat\mu_{\max}$, independently of the center
and the linearization subgradients. Choose
\begin{equation*}\label{eq:spge1}
\begin{aligned}
t_0
&:=\max\left\{
0,\,
\left\lceil\frac{2\hat\mu_{\max}}{\hat\mu_{\min}}-2
\right\rceil
\right\},\\
t_0+2
&\ge\frac{2\hat\mu_{\max}}{\hat\mu_{\min}},
\qquad
\frac{2}{\hat\mu_{\min}(t_0+2)}
\le\frac1{\hat\mu_{\max}}.
\end{aligned}
\end{equation*}
The stepsize condition in \eqref{eq:decreasing-step} therefore
holds simultaneously for all surrogate components and all
outer iterations.

We next construct a common region for the inner iterates.
At the starting point $x^k$, the composite subgradients of
all surrogate components are bounded by $M_{\rm lev}$,
so $B_0\le M_{\rm lev}$. Since
$R_S\ge2M_{\rm lev}/\hat\mu_{\min}$,
Lemma~\ref{lem:common-invariance} gives
$z^t\in\B(x^k,R_S)$ for every inner iterate. Consequently,
all inner trajectories lie in the compact set
\[
\widehat\Omega:=
\left\{v\in\R^n:
\operatorname{dist}(v,\Omega_0)\le R_S\right\}.
\]
Here compactness follows from compactness of $\Omega_0$.
Writing $R_0=P_1$ and $R_i=Q_{i1}$ for $i\in[m]$, define
\[
\hat M:=M_{\rm lev}+\hat\mu_{\max}R_S
+\max_{i\in[m]_0}
\sup_{\substack{v\in\widehat\Omega\\r\in\partial R_i(v)}}
\norm{r}.
\]
This constant is finite by local boundedness of the
subdifferentials of the finite convex functions $R_i$.
Moreover, the affine gradient formulas give, for
$v\in\B(x^k,R_S)$ and $r\in\partial R_i(v)$,
\[
\norm{\nabla h_i^k(v)+r}
\le M_{\rm lev}+\hat\mu_{\max}R_S+\norm{r}
\le\hat M,
\qquad i\in[m]_0.
\]
Thus the inner bound with radius $R=R_S$ satisfies
$M_R\le\hat M$ uniformly over all admissible centers and
linearizations.

Finally, we choose a common iteration budget. Since
$z^0=x^k$, $z^*=\hat{x}^{k+1}$, and
$\norm{\hat{x}^{k+1}-x^k}\le R_S$,
\[
A_0\le\frac{\hat\mu_{\min}}4t_0(t_0+1)R_S^2.
\]
Set
\[
C_{\rm spg}:=
\max\left\{
1,\,
R_S\sqrt{\hat\mu_{\min}t_0(t_0+1)},\,
\frac{8\hat M^2}{\hat\mu_{\min}}
\right\},
\]
and let $T_\tau$ be given by \eqref{eq:Ttau}.
For $0<\tau\le1$, we have
$1\le\tau^{-1}$ and $\tau^{-1/2}\le\tau^{-1}$, so
\begin{equation*}\label{eq:spge2}
\max\left\{
1,\,
R_S\sqrt{\frac{\hat\mu_{\min}t_0(t_0+1)}\tau},\,
\frac{8\hat M^2}{\hat\mu_{\min}\tau}
\right\}
\le\frac{C_{\rm spg}}\tau
\le T_\tau.
\end{equation*}
Together with the bounds on $A_0$ and $M_R$, and the fact
that $T_\tau$ is an integer, this verifies
\eqref{eq:inner-budget}. Theorem~\ref{thm:inner} therefore
gives \eqref{eq:uniform-oracle}.
\end{proof}

Proposition~\ref{prop:uniform-spg} provides an SPG budget of
$O(\tau^{-1})$ that controls surrogate feasibility and the
one-sided objective residual uniformly over the outer iterations
and admissible linearizations. Implementing this sufficient
budget requires available upper bounds on its constants.
We now use the uniform oracle and multiplier bound to calibrate
the outer method and show that each iteration either returns
an approximate KKT point or produces a feasible point with
sufficient descent.

% Convergence analysis: clean replacement section for the original manuscript.
\section{Convergence analysis}
\label{sec:outer}

We now use the uniform SPG oracle to analyze the outer method.
The argument proceeds from the inner error bounds to the stopping
test: we first estimate the distance to the exact surrogate solution,
then transfer its KKT system to the original problem, and finally
calibrate the computed displacement. This calibration yields an
approximate KKT guarantee when the method stops and sufficient
descent whenever it continues.

\begin{proposition}[Distance to the exact surrogate solution]
%\label{prop:inexact-distance}
Suppose that Assumptions~\ref{ass:data} and \ref{ass:slater} hold.
Let $x^k\in\Omega_0$, let $0<\tau\le1$, and let
$x^{k+1}$ be the SPG output in Proposition~\ref{prop:uniform-spg}.
Then
\begin{equation}\label{eq:oracle-distance}
\norm{x^{k+1}-\hat{x}^{k+1}}
\le\sqrt{\frac{2(1+\Lambda)\tau}{\mu_f}}.
\end{equation}
\end{proposition}

\begin{proof}
Fix a surrogate KKT multiplier $\lambda_k$ from
Proposition~\ref{prop:multiplier-bound} and define
\[
\cL_k(v):=\Phi_k(v)
+\sum_{i=1}^m\lambda_{i,k}\Psi_{i,k}(v).
\]
Since the surrogate functions are finite convex, the exact
subdifferential sum rule and the KKT system give
$0\in\partial\cL_k(\hat{x}^{k+1})$. Moreover, $\cL_k$ is
$\mu_f+\sum_i\lambda_{i,k}\mu_{g_i}$-strongly convex.
Thus $\hat{x}^{k+1}$ is its unique minimizer. The resulting sharper
estimate is
\[
\norm{x^{k+1}-\hat{x}^{k+1}}^2
\le\frac{2(1+\norm{\lambda_k}_1)\tau}
{\mu_f+\sum_i\lambda_{i,k}\mu_{g_i}}.
\]
For the uniform bound below, it suffices to use the modulus $\mu_f$. Applying the strong-convexity inequality at
$x^{k+1}$ and using complementarity, we obtain
\[
\begin{aligned}
\frac{\mu_f}{2}\norm{x^{k+1}-\hat{x}^{k+1}}^2
&\le\cL_k(x^{k+1})-\cL_k(\hat{x}^{k+1})\\
&=\Phi_k(x^{k+1})-\Phi_k(\hat{x}^{k+1})
+\sum_{i=1}^m\lambda_{i,k}\Psi_{i,k}(x^{k+1})\\
&\le\bigl(1+\norm{\lambda_k}_1\bigr)\tau
\le(1+\Lambda)\tau.
\end{aligned}
\]
The last line uses \eqref{eq:uniform-oracle}, nonnegativity
of $\lambda_k$, and \eqref{eq:Mlambda}. In particular,
the argument requires only the one-sided objective bound.
Taking square roots proves \eqref{eq:oracle-distance}.
\end{proof}

The preceding estimate links the computed displacement to the
exact surrogate displacement. We next determine how small the
latter must be to certify the current center. The exact solution
and its multiplier will serve as an analytical witness for the
split-point condition in Definition~\ref{def:near-kkt}.

To obtain a uniform transfer bound, define
\[
\begin{gathered}
M_{Q_2}:=
\max_{i\in[m]}
\sup_{\substack{x\in\Omega_0\\q\in\partial Q_{i2}(x)}}
\norm{q}<+\infty,\\
\hat{\rho}_{\rm dc}:=
\max\left\{
1,\,
L_f+\mu_f,\,
\max_{i\in[m]}(L_{g_i}+\mu_{g_i}),\,
\max_{i\in[m]}
\left(\frac{L_{g_i}+\mu_{g_i}}2+2M_{Q_2}\right)
\right\}.
\end{gathered}
\]
Here $M_{Q_2}$ is finite by compactness of $\Omega_0$ and
local boundedness of the subdifferentials, while
$L_f+\mu_f=2L_f+\rho_f$ and
$L_{g_i}+\mu_{g_i}=2L_{g_i}+\rho_{g_i}$.
The constant $\hat{\rho}_{\rm dc}$ will control both the
stationarity residual and the constraint model gaps.
Let $\rho_{\min}:=\min_{i\in[m]}\rho_{g_i}>0$ and fix
\begin{equation}\label{eq:fixed-calibration}
\begin{aligned}
\kappa_{\rm tr}
&\ge\hat{\rho}_{\rm dc}(1+\Lambda),\\
0<c_{\rm tr}
&\le\min\left\{
1,\,
\frac{\mu_f}{32(1+\Lambda)\kappa_{\rm tr}^2},\,
\frac{\rho_{\min}}{4\kappa_{\rm tr}^2}
\right\}.
\end{aligned}
\end{equation}
The first constant gives a uniform displacement threshold
for KKT transfer; the second will set the inner accuracy
needed for the stopping and acceptance guarantees.
Both are independent of the outer index and the target accuracy.

\begin{lemma}[KKT transfer from the exact surrogate]
\label{lem:kkt-transfer}
Suppose that Assumptions~\ref{ass:data} and \ref{ass:slater} hold.
Let $x^k\in\Omega_0$, let $\hat{x}^{k+1}$ solve
\eqref{eq:surrogate}, and let
$(\lambda_k,p_1^k,(q_{i1}^k)_{i=1}^m)$ satisfy
\eqref{eq:surrogate-KKT-expanded} as supplied by
Proposition~\ref{prop:multiplier-bound}; in particular,
$\norm{\lambda_k}_1\le\Lambda$. For any $\eps\in(0,1]$, if
\begin{equation}\label{eq:kkt-transfer-condition}
\norm{\hat{x}^{k+1}-x^k}
\le\frac{\eps}
{\hat{\rho}_{\rm dc}(1+\norm{\lambda_k}_1)},
\end{equation}
then $x^k$ is an $\eps$-KKT point of \eqref{eq:problem} in
the sense of Definition~\ref{def:near-kkt}. In particular,
since $\norm{\lambda_k}_1\le\Lambda$, the uniform condition
$\norm{\hat{x}^{k+1}-x^k}\le\eps/\kappa_{\rm tr}$ is sufficient.
Consequently, if $x^k$ is not an $\eps$-KKT point, then
$\norm{\hat{x}^{k+1}-x^k}>\eps/\kappa_{\rm tr}$.
\end{lemma}

\begin{proof}
Set $d^k:=\hat{x}^{k+1}-x^k$. By
Lemma~\ref{lem:surrogate-localization}, both $x^k$ and
$\hat{x}^{k+1}$ belong to $\Omega_0$. We use
$(x,z,\lambda)=(x^k,\hat{x}^{k+1},\lambda_k)$ as the witness
in Definition~\ref{def:near-kkt} and verify proximity,
stationarity, and complementarity in turn.

First, \eqref{eq:kkt-transfer-condition} and
$\hat{\rho}_{\rm dc}\ge1$ imply
$\norm{d^k}\le\eps\le1$. For stationarity, define
\[
\begin{aligned}
r^k:={}&\nabla f(\hat{x}^{k+1})+p_1^k-\xi^k\\
&+\sum_{i=1}^m\lambda_{i,k}
\bigl(\nabla g_i(\hat{x}^{k+1})+q_{i1}^k-\zeta_i^k\bigr).
\end{aligned}
\]
Subtracting the stationarity identity in
\eqref{eq:surrogate-KKT-expanded} gives
\[
\begin{aligned}
r^k={}&\nabla f(\hat{x}^{k+1})-\nabla f(x^k)-\mu_fd^k\\
&+\sum_{i=1}^m\lambda_{i,k}
\bigl(\nabla g_i(\hat{x}^{k+1})-\nabla g_i(x^k)
-\mu_{g_i}d^k\bigr).
\end{aligned}
\]
Lipschitz continuity of the gradients therefore yields
\[
\begin{aligned}
\norm{r^k}
&\le\left[L_f+\mu_f+
\sum_{i=1}^m\lambda_{i,k}(L_{g_i}+\mu_{g_i})\right]
\norm{d^k}\\
&\le\hat{\rho}_{\rm dc}(1+\norm{\lambda_k}_1)
\norm{d^k}\le\eps.
\end{aligned}
\]
Here $p_1^k$ and $q_{i1}^k$ are subgradients at
$\hat{x}^{k+1}$, whereas $\xi^k$ and $\zeta_i^k$ are
subgradients at $x^k$. Thus $r^k$ is the split-point
stationarity residual required by Definition~\ref{def:near-kkt}.

To transfer complementarity, we first bound the constraint
model gaps. For each $i\in[m]$, choose
$\bar\zeta_i^k\in\partial Q_{i2}(\hat{x}^{k+1})$.
Since both endpoints belong to $\Omega_0$,
$\norm{\zeta_i^k},\norm{\bar\zeta_i^k}\le M_{Q_2}$.
Expanding the gap and using the lower smoothness estimate for
$g_i$ and the subgradient inequality for $Q_{i2}$ at
$\hat{x}^{k+1}$, we obtain
\[
\begin{aligned}
0
&\le\Psi_{i,k}(\hat{x}^{k+1})-G_i(\hat{x}^{k+1})\\
&=g_i(x^k)+\ip{\nabla g_i(x^k)}{d^k}
-g_i(\hat{x}^{k+1})
+\frac{\mu_{g_i}}2\norm{d^k}^2\\
&\quad+Q_{i2}(\hat{x}^{k+1})-Q_{i2}(x^k)
-\ip{\zeta_i^k}{d^k}\\
&\le\frac{L_{g_i}+\mu_{g_i}}2\norm{d^k}^2
+\ip{\bar\zeta_i^k-\zeta_i^k}{d^k}\\
&\le\frac{L_{g_i}+\mu_{g_i}}2\norm{d^k}^2
+2M_{Q_2}\norm{d^k}\\
&\le\hat{\rho}_{\rm dc}\norm{d^k},
\qquad i\in[m].
\end{aligned}
\]
The first inequality follows from constraint majorization,
and the last uses $\norm{d^k}\le1$ and the definition of
$\hat{\rho}_{\rm dc}$. This estimate requires only bounded
subgradients, not continuity of the selections $\zeta_i(\cdot)$.

Finally, surrogate complementarity and feasibility of
$\hat{x}^{k+1}$ for the original problem give
\[
\begin{aligned}
\sum_{i=1}^m\left|\lambda_{i,k}G_i(\hat{x}^{k+1})\right|
&=\sum_{i=1}^m\lambda_{i,k}
\left[\Psi_{i,k}(\hat{x}^{k+1})-G_i(\hat{x}^{k+1})\right]\\
&\le\hat{\rho}_{\rm dc}\norm{\lambda_k}_1\norm{d^k}\\
&\le\hat{\rho}_{\rm dc}(1+\norm{\lambda_k}_1)
\norm{d^k}\le\eps.
\end{aligned}
\]
Together with feasibility, proximity, and stationarity, this
verifies Definition~\ref{def:near-kkt}. Moreover,
$\kappa_{\rm tr}\ge\hat{\rho}_{\rm dc}(1+\Lambda)
\ge\hat{\rho}_{\rm dc}(1+\norm{\lambda_k}_1)$, so
$\norm{d^k}\le\eps/\kappa_{\rm tr}$ implies
\eqref{eq:kkt-transfer-condition}. The last assertion follows
by contraposition.
\end{proof}

We now calibrate the computed stopping test. To apply
Lemma~\ref{lem:kkt-transfer}, we reserve one quarter of its
uniform displacement threshold for the inner error and use
the remaining three quarters for the computed displacement.
For a target accuracy $\eps\in(0,1]$, set
\begin{equation}\label{eq:mgspg-calibration}
\eta_\eps:=\frac{3\eps}{4\kappa_{\rm tr}},
\qquad
\tau_\eps:=c_{\rm tr}\eps^2,
\qquad
T_\eps:=T_{\tau_\eps}
=\left\lceil\frac{C_{\rm spg}}{c_{\rm tr}\eps^2}\right\rceil.
\end{equation}
The constants in \eqref{eq:fixed-calibration} remain fixed,
while these three parameters vary with $\eps$. The next
theorem shows that this choice both validates the stopping
test and guarantees acceptance when the test fails.

\begin{theorem}[Approximate KKT guarantee or sufficient descent]
\label{thm:one-step}
Suppose that Assumptions~\ref{ass:data} and \ref{ass:slater} hold.
Let $x^k\in\Omega_0$ and $\eps\in(0,1]$, and choose the
parameters according to
\eqref{eq:fixed-calibration}--\eqref{eq:mgspg-calibration},
using the common $t_0$ from Proposition~\ref{prop:uniform-spg}.
Then the prescribed inner budget produces an output, and the
following conclusions hold.
\begin{enumerate}[{\rm (i)}]
\item If $\norm{x^{k+1}-x^k}\le3\eps/(4\kappa_{\rm tr})$,
then $x^k$ is an $\eps$-KKT point.

\item If $\norm{x^{k+1}-x^k}>3\eps/(4\kappa_{\rm tr})$,
then
\begin{equation}\label{eq:strict-feasibility}
G_i(x^{k+1})
<-\frac{\eps^2}{32\kappa_{\rm tr}^2}
\min_{j\in[m]}\rho_{g_j}<0,
\qquad i\in[m].
\end{equation}
Moreover,
\begin{equation}\label{eq:accepted-descent}
F(x^{k+1})
\le F(x^k)-\frac{\mu_f}{16\kappa_{\rm tr}^2}\eps^2.
\end{equation}
Consequently, the output passes the acceptance safeguard,
$x^{k+1}\in\Omega_0$, and the algorithm proceeds to the next
outer iteration.
\end{enumerate}
\end{theorem}

\begin{proof}
Since $c_{\rm tr}\le1$ and $\eps\in(0,1]$, we have
$0<\tau_\eps\le1$. Proposition~\ref{prop:uniform-spg}
therefore guarantees an inner output. Furthermore,
\eqref{eq:oracle-distance} and the bound on $c_{\rm tr}$ in
\eqref{eq:fixed-calibration} give
\begin{equation}\label{eq:one-step-inner-distance}
\norm{x^{k+1}-\hat{x}^{k+1}}
\le\sqrt{\frac{2(1+\Lambda)\tau_\eps}{\mu_f}}
=\sqrt{\frac{2(1+\Lambda)c_{\rm tr}}{\mu_f}}\,\eps
\le\frac{\eps}{4\kappa_{\rm tr}}.
\end{equation}

We first prove~(i). If the stopping test holds, the triangle
inequality and \eqref{eq:one-step-inner-distance} yield
\[
\begin{aligned}
\norm{\hat{x}^{k+1}-x^k}
&\le\norm{\hat{x}^{k+1}-x^{k+1}}
+\norm{x^{k+1}-x^k}\\
&\le\frac{\eps}{4\kappa_{\rm tr}}
+\frac{3\eps}{4\kappa_{\rm tr}}
=\frac{\eps}{\kappa_{\rm tr}}.
\end{aligned}
\]
Lemma~\ref{lem:kkt-transfer} therefore shows that $x^k$ is
an $\eps$-KKT point.

We now turn to~(ii). Suppose that
$\norm{x^{k+1}-x^k}>3\eps/(4\kappa_{\rm tr})$.
Constraint majorization and the inner feasibility bound imply,
for every $i\in[m]$,
\[
\begin{aligned}
G_i(x^{k+1})
&\le\Psi_{i,k}(x^{k+1})
-\frac{\rho_{g_i}}2\norm{x^{k+1}-x^k}^2\\
&<c_{\rm tr}\eps^2
-\frac{9\rho_{g_i}}{32\kappa_{\rm tr}^2}\eps^2\\
&\le\frac{8\rho_{\min}-9\rho_{g_i}}
{32\kappa_{\rm tr}^2}\eps^2
\le-\frac{\rho_{\min}}{32\kappa_{\rm tr}^2}\eps^2<0.
\end{aligned}
\]
Here we used $c_{\rm tr}\le\rho_{\min}/(4\kappa_{\rm tr}^2)$
and $\rho_{g_i}\ge\rho_{\min}$. This proves
\eqref{eq:strict-feasibility}.

To obtain descent, observe that the reverse triangle inequality
and \eqref{eq:one-step-inner-distance} give
\[
\begin{aligned}
\norm{\hat{x}^{k+1}-x^k}
&\ge\norm{x^{k+1}-x^k}
-\norm{x^{k+1}-\hat{x}^{k+1}}\\
&>\frac{3\eps}{4\kappa_{\rm tr}}
-\frac{\eps}{4\kappa_{\rm tr}}
=\frac{\eps}{2\kappa_{\rm tr}}.
\end{aligned}
\]
Combining this bound with objective majorization,
\eqref{eq:uniform-oracle}, and
\eqref{eq:strong-surrogate-descent}, we obtain
\[
\begin{aligned}
F(x^{k+1})
&\le\Phi_k(x^{k+1})
\le\Phi_k(\hat{x}^{k+1})+c_{\rm tr}\eps^2\\
&\le F(x^k)
-\frac{\mu_f}{2}\norm{\hat{x}^{k+1}-x^k}^2
+c_{\rm tr}\eps^2\\
&<F(x^k)-\frac{\mu_f}{8\kappa_{\rm tr}^2}\eps^2
+\frac{\mu_f}{32\kappa_{\rm tr}^2}\eps^2\\
&=F(x^k)-\frac{3\mu_f}{32\kappa_{\rm tr}^2}\eps^2
\le F(x^k)-\frac{\mu_f}{16\kappa_{\rm tr}^2}\eps^2,
\end{aligned}
\]
where the strict inequality also uses
$c_{\rm tr}\le\mu_f/[32(1+\Lambda)\kappa_{\rm tr}^2]
\le\mu_f/(32\kappa_{\rm tr}^2)$.
Thus \eqref{eq:accepted-descent} holds. Together with strict
feasibility, it ensures acceptance and gives
$x^{k+1}\in\Omega_0$.
\end{proof}

\begin{remark}[Computed and auxiliary quantities]
%\label{rem:computed-auxiliary}
The stopping test uses only the computed displacement
$\norm{x^{k+1}-x^k}$ and returns the feasible center $x^k$.
The exact solution $\hat{x}^{k+1}$, multiplier $\lambda_k$,
and subgradients $p_1^k,q_{i1}^k$ enter only the analytical
witness. Thus executing the test requires neither an exact
surrogate solve nor a multiplier computation. Its guarantee
is the split-point condition in Definition~\ref{def:near-kkt};
a residual with all subgradients evaluated at the returned
point requires additional regularity. The threshold still
depends on the uniform constants in
\eqref{eq:fixed-calibration}. When only their existence is
known, \eqref{eq:mgspg-calibration} is an a priori
oracle-complexity prescription.
\end{remark}

The one-step alternative now gives a finite termination bound.
Whenever the stopping test fails, the objective decreases
by at least a fixed multiple of $\eps^2$. Since it is bounded
below on $\Omega_0$, only finitely many such steps are possible.
Combining this outer count with the common SPG budget yields
the total oracle cost.

\begin{theorem}[Finite-accuracy and oracle complexity]
%\label{thm:outer-complexity}
Suppose that Assumptions~\ref{ass:data} and \ref{ass:slater} hold.
For any $\eps\in(0,1]$, run Algorithm~\ref{alg:mgspg} with
\eqref{eq:fixed-calibration}--\eqref{eq:mgspg-calibration}
and the common $t_0$ from Proposition~\ref{prop:uniform-spg}.
Then the method returns an $\eps$-KKT point within
\[
 K_\eps:=
 \left\lfloor
 \frac{16\kappa_{\rm tr}^2\bigl(F(x^0)-F_{\inf}\bigr)}
      {\mu_f\eps^2}
 \right\rfloor+1
 =O(\eps^{-2})
\]
outer iterations, without reporting
``insufficient inner accuracy''.
The multiplier in the $\eps$-KKT witness can be chosen
with $\ell_1$ norm at most $\Lambda$.

Moreover, each inner call uses $T_\eps$ selected component proximal
gradient updates, where $T_\eps$ is the integer defined in
\eqref{eq:mgspg-calibration}. In particular,
$T_\eps\le C_{\rm spg}/(c_{\rm tr}\eps^2)+1$.
Consequently, such a point is found using at most
\begin{equation}\label{eq:total-product-bound}
 K_\eps T_\eps
 \leq
 \left(
 \frac{16\kappa_{\rm tr}^2\bigl(F(x^0)-F_{\inf}\bigr)}
      {\mu_f\eps^2}+1
 \right)
 \left(
 \frac{C_{\rm spg}}{c_{\rm tr}\eps^2}+1
 \right)
 =O(\eps^{-4})
\end{equation}
selected component proximal gradient updates.
With full constraint scans, the number of surrogate
constraint-value evaluations is $O(m\eps^{-4})$.
\end{theorem}

\begin{proof}
By Proposition~\ref{prop:uniform-spg} and
Theorem~\ref{thm:one-step}, induction from $x^0\in\Omega_0$
excludes the report ``insufficient inner accuracy''.
Every nonterminating outer iteration satisfies
\[
 x^{k+1}\in\Omega_0,\qquad
 F(x^{k+1})
 \leq F(x^k)-\frac{\mu_f\eps^2}{16\kappa_{\rm tr}^2}.
\]
Thus, if the first $K_\eps$ iterations were nonterminating,
\[
 F_{\inf}
 \leq F(x^{K_\eps})
 \leq F(x^0)
      -\frac{K_\eps\mu_f\eps^2}{16\kappa_{\rm tr}^2}
 <F_{\inf},
\]
a contradiction. Hence the method returns within $K_\eps$
iterations, and~\eqref{eq:Mlambda} gives the multiplier bound.

Next, let $K\leq K_\eps$ be the number of outer iterations
performed, including the terminating iteration, and let $T_k$
be the number of selected updates in its $k$th inner call.
By~\eqref{eq:Ttau}, with $\tau=c_{\rm tr}\eps^2$,
\[
 T_k=T_\eps=\left\lceil\frac{C_{\rm spg}}{c_{\rm tr}\eps^2}\right\rceil
 \leq\frac{C_{\rm spg}}{c_{\rm tr}\eps^2}+1,
 \qquad k=0,\ldots,K-1.
\]
Consequently,
\[
 \sum_{k=0}^{K-1}T_k
 \leq KT_\eps
 \leq K_\eps T_\eps.
\]
Together with the definition of $K_\eps$, this proves
\eqref{eq:total-product-bound}.Each selected component proximal gradient update uses one affine
surrogate gradient and one individual proximal mapping, so each
of these counts is at most $K_\eps T_\eps$. Full constraint scans use
$O(mK_\eps T_\eps)=O(m\eps^{-4})$ surrogate constraint values.
\end{proof}

% Standalone numerical-experiments section.
% Required packages: amsmath, booktabs, graphicx, multirow.
% This source was not compiled as part of the experiment.

% Standalone numerical-experiments section.
% Required packages: amsmath, amssymb, booktabs, graphicx.
% This file is intentionally not a complete LaTeX document.

% Final numerical-experiments section.
% Required packages: amsmath, amssymb, booktabs, graphicx.
% This file is intentionally not a complete LaTeX document.
%
% If this folder is next to the main manuscript, place
%   \renewcommand{\MGSPGExpPath}{MGSPG_final_experiments/}
% before \input{MGSPG_final_experiments/MGSPG_Numerical_Experiments.tex}.

\section{Numerical experiments}\label{sec:numerical}

In this section, we benchmark the proposed MGSPG (Algorithm~2) against PGSSM~\cite[Algorithm~3.2]{JiaGrimmer2025} on a SCAD-constrained regression problem (Section~\ref{subsec:scad}) and against IDCA--SSG~\cite{HeHuangYaoLin2025} on the groupwise-constrained model of Example~1.1 (Section~\ref{subsec:groupwise}). All algorithms are implemented in \textsc{Matlab} (Release R2025b) and executed on a PC equipped with an Intel Core i7-8550U CPU ($1.80\,\mathrm{GHz}$) and $16\,\mathrm{GB}$ of RAM.

\subsection{Evaluation criterion and execution metrics}%\label{subsec:criterion}
\noindent\textit{Evaluation.}
For $G(x) := (G_1(x), \dots, G_m(x))^\top$ and $\lambda \ge 0$, set $H_\lambda := P_1 + \sum_i \lambda_i Q_{i1}$ and introduce the affine-in-$\lambda$ vector
\[
 a(x, \lambda) := \nabla f(x) - \xi(x) + \sum_{i=1}^m \lambda_i \bigl(\nabla g_i(x) - \zeta_i(x)\bigr),
\]
where $\xi(x) = \nabla P_2(x)$ and $\zeta_i(x) = \nabla Q_{i2}(x)$ at points of differentiability. When $Q_{i2}(x) = \alpha_i \|D_i x\|_2$, we adopt the canonical subgradient selection $\zeta_i(x) := \alpha_i D_i^\top D_i x / \|D_i x\|_2$ if $D_i x \neq 0$, and $\zeta_i(x) := 0$ otherwise. To assess first-order criticality and complementary slackness, we define the normalized stationarity--complementarity residual
\[
 R_{\mathrm{sc}}(x) := \min_{0\le\lambda\le\Lambda_{\max}} \max \left\{ \frac{\|x - \operatorname{prox}_{H_\lambda}(x - a(x,\lambda))\|}{1 + \|x\|}, \; \frac{\|\lambda \odot G(x)\|_\infty}{1 + \|\lambda\|_\infty} \right\}.
\]

Here, $\odot$ denotes the Hadamard product. Let $\widehat{R}_{\mathrm{sc}}(x)$ denote the smallest residual obtained from admissible multiplier candidates. We evaluate numerical KKT accuracy using
\[
 R_{\mathrm{KKT}}^{\mathrm{num}}(x) := \max \left\{ \frac{\widehat{R}_{\mathrm{sc}}(x)}{\epsilon_{\mathrm{KKT}}}, \; \frac{\|[G(x)]_+\|_\infty}{\epsilon_{\mathrm{feas}}} \right\}, \quad \epsilon_{\mathrm{KKT}} = \epsilon_{\mathrm{feas}} = 10^{-2}.
\]
The common accuracy criterion is $R_{\mathrm{KKT}}^{\mathrm{num}}(x) \le 1$. For the diagnostic index set $\mathcal K_{\mathrm{diag}}$, efficiency is measured by the first observed attainment time
\[
 t_{\mathrm{hit}}^{\mathrm{wall}}
 := \min\bigl\{t_k^{\mathrm{wall}} : k\in\mathcal K_{\mathrm{diag}},\,
 R_{\mathrm{KKT}}^{\mathrm{num}}(x^k)\le1\bigr\},
 \quad \min\varnothing:=+\infty.
\]

\smallskip\noindent\textit{Experimental protocol.}
In each comparison, both methods are tested on the same 50 instances and initialized at $x^0=0$. We define $B_k:=\sum_{j=0}^{k-1}T_j$, where $T_j$ is the prescribed inner-iteration limit at outer iteration $j$. Each inner solve is counted as $T_j$ iterations, even if stopped early. At outer iteration \(k\), the inner algorithm is executed only if \(B_k+T_k\le B_{\max}\). For the SCAD problem, the displacement stopping criterion $\|x^{k+1}-x^k\|\le\eta$ is disabled to facilitate fixed-budget comparisons. In contrast, for the groupwise model, MGSPG retains this condition alongside the feasibility and descent checks in Algorithm~2, whereas IDCA--SSG uses $B_{\max}$ as its iteration limit.

The numerical KKT residuals are evaluated offline and are not used in the stopping criteria. They are evaluated at every accepted SCAD iterate and, for the groupwise model, at the last accepted iterate at or before each increment of 1000 in $B_k$, including the initial and terminal iterates. The wall-clock times in Table~\ref{tab:groupwise} and Figs.~\ref{fig:scad-perf}--\ref{fig:groupwise-scaling} exclude data generation, file input/output, and KKT residual evaluation.

\subsection{SCAD-constrained regression}\label{subsec:scad}
Inspired by the example in \cite[Section~5]{JiaGrimmer2025}, we replace the nonsmooth loss therein with the smooth Cauchy loss $\ell_\kappa(t) := \frac{\kappa^2}{2}\log(1 + t^2/\kappa^2)$ and consider the following SCAD-constrained regression problem
\begin{equation*}\label{eq:scad-prob}
    \min_{x\in\mathbb{R}^n} \frac{1}{N}\sum_{r=1}^N \ell_\kappa(a_r^\top x - b_r) + \lambda_{\mathrm{reg}}\|x\|_1 \quad \text{s.t.} \quad \sum_{j=1}^n s(x_j) \le p,
\end{equation*}
where $s(t) := 2|t|$ for $|t|\le 1$, $s(t) := -t^2+4|t|-1$ for $1 < |t|\le 2$, and $s(t) := 3$ for $|t|>2$.

For this problem, we set $N = n = 120$, $\kappa = 1$, and $\lambda_{\mathrm{reg}} = 10^{-3}$. The sensing matrix $A = [a_1, \dots, a_N]^\top$ has i.i.d.\ $\mathcal{N}(0,1)$ entries, $x^\star$ contains $30$ non-zero components drawn uniformly from $[-10,-5]\cup[5,10]$, and the observation vector satisfies $b = Ax^\star + \varepsilon$ with $\varepsilon \sim \mathcal{N}(0,I)$. The bounds $p \in \{91, 320\}$ correspond to nearly active and inactive regimes, respectively, while avoiding the degenerate points $p \in 3\mathbb{Z}$.

\subsubsection{Parameter settings}
For the SCAD constraint, we use $g_1(x)\equiv-p$, $Q_{11}(x)=2\|x\|_1$, and $Q_{12}(x)=\sum_{j=1}^n q(x_j)$, where
\[
q(t)=
\begin{cases}
0, & |t|\le1,\\
(|t|-1)^2, & 1<|t|\le2,\\
2|t|-3, & |t|>2.
\end{cases}
\]
Then $s(t)=2|t|-q(t)$ and $L_{g_1}=0$, so $\mu_g=\rho_g$.

The parameters for MGSPG are set to $(\rho_f, \rho_g, \tau, T, \eta) = (0.1, 2.97, 10^{-4}, 10^4, 10^{-3})$, with $\mu_f = L_f + \rho_f$, $\mu_g = \rho_g$, $L_f = \|A\|_2^2/N$, and $t_0$ chosen as the smallest nonnegative integer satisfying the stepsize condition in Algorithm~\ref{alg:mgspg}. These practical choices are selected independently of the theoretical calibration in \eqref{eq:fixed-calibration} and \eqref{eq:mgspg-calibration}. PGSSM is implemented following~\cite[Algorithm~3.2]{JiaGrimmer2025} with $\epsilon = 10^{-2}$, $T = 10^4$, $\rho = \max\{3, L_f/8 + 10^{-6}\}$, $\hat{\rho} = 2\rho$, $\operatorname{sign}(0) = 0$, and the identity map replacing the box projection.

Both algorithms share a budget cap of $B_{\max} = 10^7$. Scaling tests are conducted over $10$ paired trials per dimension $N = n \in \{200, 400, 600, 800\}$, each employing $10^5$ inner updates, support size $n/4$, and $p = 3n/4 + 1$. For $P_1 = \lambda_{\mathrm{reg}}\|\cdot\|_1$ and $Q_{11} = 2\|\cdot\|_1$, we define $[\mathcal{S}_\theta(v)]_j := \operatorname{sign}(v_j)\max\{|v_j|-\theta, 0\}$. The inner-loop iterations and offline KKT evaluations utilize
\[
\operatorname{prox}_{\gamma P_1} = \mathcal{S}_{\gamma\lambda_{\mathrm{reg}}}, \quad
\operatorname{prox}_{\gamma Q_{11}} = \mathcal{S}_{2\gamma}, \quad
\operatorname{prox}_{H_\lambda} = \mathcal{S}_{\lambda_{\mathrm{reg}} + 2\lambda},
\]
where $\gamma > 0$ denotes the stepsize and $\lambda \ge 0$ represents the constraint multiplier.

\FloatBarrier
\subsubsection{Numerical results for SCAD-constrained regression}
Across all 50 trials in both regimes, both solvers reached the benchmark $R_{\mathrm{KKT}}^{\mathrm{num}} \le 1$. However, MGSPG required $\mathbf{33-34}$\% fewer inner updates and $\mathbf{58-59}$\% less wall-clock time, yielding an overall $\mathbf{2.4}$-fold speedup (Figure~\ref{fig:scad-perf}e,f). Under the full budget $B_{\max} = 10^7$, median runtimes were $153\,\mathrm{s}$ for MGSPG versus $254\,\mathrm{s}$ for PGSSM (a 1.66-fold gain). Over the common horizon ($\approx 150\,\mathrm{s}$), MGSPG attained 8.1\% ($p = 91$) and 13.6\% ($p = 320$) lower objective values alongside 49--57\% smaller $R_{\mathrm{KKT}}^{\mathrm{num}}$ residuals. 

\begin{herefigure}
\centering
\includegraphics[width=\textwidth,clip]{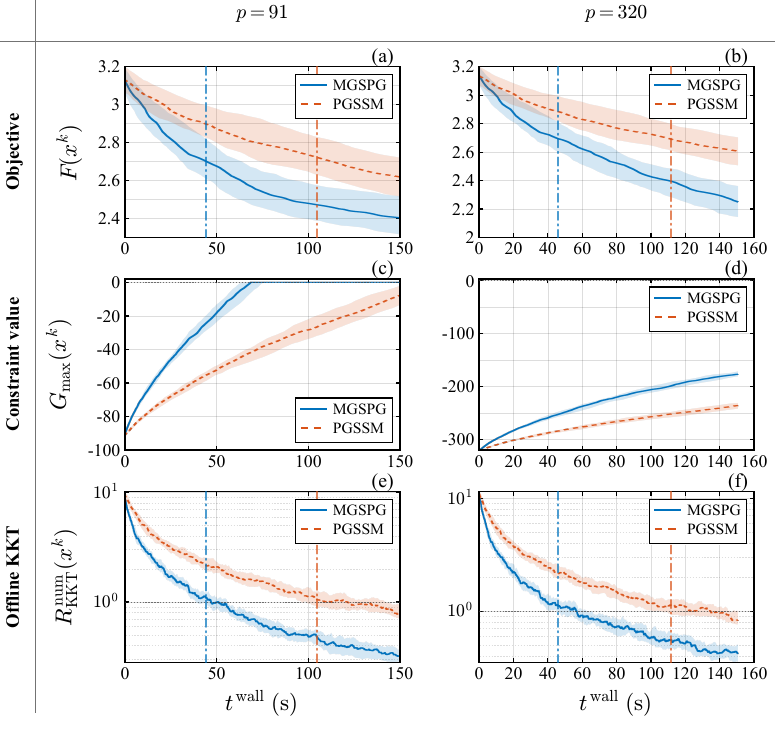}
% Plot construction is unchanged: last observations are carried forward to the earliest terminal time.
\captionof{figure}{Performance comparison of MGSPG and PGSSM for SCAD-constrained regression. Curves and bands show medians and interquartile ranges over 50 paired instances. Dash-dotted vertical lines mark median hit times.}
\label{fig:scad-perf}
\end{herefigure}

While both methods maintained strictly feasible outer iterates, MGSPG identified the active constraint boundary substantially earlier when $p = 91$, whereas the constraint remained inactive for $p = 320$. This computational advantage is even more pronounced in the scaling study ($10^5$ updates, $N = n$ increasing from 200 to 800). The runtime of MGSPG grew by only $\mathbf{35.9}$\% ($1.64\,\mathrm{s}$ to $2.22\,\mathrm{s}$) compared to a $\mathbf{420.2}$\% surge ($2.89\,\mathrm{s}$ to $15.04\,\mathrm{s}$) for PGSSM, widening the speedup factor from 1.77 to 6.76 (Figure~\ref{fig:scad-cost}b) and demonstrating superior scalability on high-dimensional problems.

\begin{herefigure}
\centering
\includegraphics[width=\textwidth,clip]{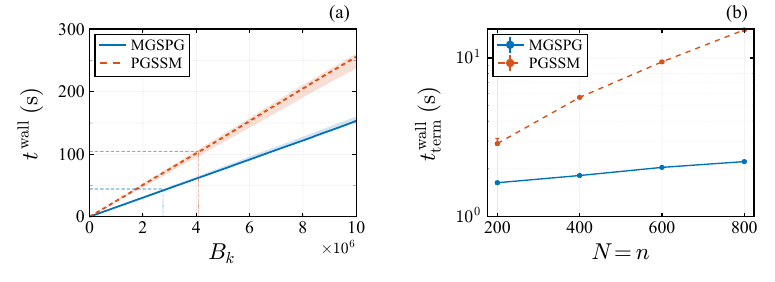}
\captionof{figure}{SCAD runtime comparison versus (a) budget ($p=91$) and (b) dimension ($10^5$ updates). Dashed lines in (a) mark median hit times and budgets.}
\label{fig:scad-cost}
\end{herefigure}
\FloatBarrier

\subsection{Robust regression with groupwise validation constraints}\label{subsec:groupwise}
Numerical experiments are conducted on the model in Example~\ref{ex:application-regression} with $n = N_0 \in \{60, 120, 240\}$, $m \in \{n/2, n, 2n\}$, and $N_i = \lceil n/3 \rceil$. Both solvers are evaluated on the same $50$ instances for each pair of $(n, m)$. The sensing matrices $A$ and $C_i$ are generated with i.i.d.\ standard Gaussian entries and subsequently column-normalized to Euclidean norms of $\sqrt{N_0}$ and $\sqrt{N_i}$, respectively.

The ground-truth signal is defined as $x^\star = 5\sqrt{N_0} v / \|Av\|$, where $v \in \{-1, 0, 1\}^n$ has a support $S$ of cardinality $30$ and independent random signs. For each group, a three-element subset $B_i \subset S$ is sampled uniformly, and we set $x_i^\star = x^\star - 2 P_{S \setminus B_i} x^\star$, which reverses the signs on $S \setminus B_i$.

Observations are generated according to $b = A x^\star + \varepsilon_0 + \eta_0$ and $d_i = C_i x_i^\star + \varepsilon_i + \eta_i$, where $\varepsilon_0$ and $\varepsilon_i$ are standard Gaussian noise vectors. The sparse outlier vectors $\eta_0$ and $\eta_i$ contain $12$ and $1 + \operatorname{mod}(i-1, 10)$ non-zero entries, respectively, each having a magnitude of $10$ with a random sign. The diagonal weights of each $D_i$ are sampled independently from $\operatorname{Uniform}(0.5, 1.5)$ and rescaled to have a sample mean of one.

\subsubsection{Parameter settings}
We set $\kappa_i = 5$ ($i = 0, \ldots, m$), $\beta = 1$, $\lambda_{\mathrm{reg}} = 0.1$, and $a_{\mathrm{mcp}} = 3$. For $i = 1, \ldots, m$, let $\mathcal{L}_i(x) := N_i^{-1} \sum_{r=1}^{N_i} \ell_{\kappa_i}(c_{ir}^\top x - d_{ir})$. The constraint parameters are defined by
\begin{equation*}
\nu_i = \mathcal{L}_i(0) + \frac{L_{g_i}F(0)}{8\beta} + \bar\delta,
\quad
\alpha_i = \vartheta \frac{\mathcal{L}_i(0)}{\|D_i x^\star\|_1 - \|D_i x^\star\|_2},
\end{equation*}
where $L_{g_i} = \|C_i\|_2^2 / N_i$ and $\bar\delta = 10^{-4}$. Since $G_i(0) = -L_{g_i}F(0)/(8\beta) - \bar\delta < 0$, the initial point $x^0 = 0$ is strictly feasible. We fix $\vartheta = 2$ for all nine $(n, m)$ pairs, following a selection from $\{1, 2, 3\}$ in an independent pilot study ($n = 120$, $m = 10$, $T = 10^4$). This selection relies on neither terminal objective values nor KKT residuals.

%\Needspace{5\baselineskip}
Both methods use the same DC decomposition and quadratic surrogates, with $\rho_f = \rho_{g_i} = 0.1$, $\mu_f = L_f + \rho_f$, and $\mu_{g_i} = L_{g_i} + \rho_{g_i}$, where $L_f = \|A\|_2^2 / N_0$. The common budget is set to $B_{\max} = 2 \times 10^5$. MGSPG uses $(\tau, T, \eta) = (10^{-4}, 2000, 10^{-3})$, with $t_0$ chosen as the smallest nonnegative integer satisfying the stepsize condition in Algorithm~\ref{alg:mgspg} and a tolerance of $10^{-10}$ for the feasibility and descent checks. IDCA--SSG uses $(\epsilon, T) = (10^{-3}, 200)$, a surrogate feasibility tolerance of $\epsilon + 10^{-8}$, and the identity projection onto $\mathbb{R}^n$.

For $D_i = \operatorname{diag}(w_{i1}, \ldots, w_{in})$, the convex terms $P_1(x) = \lambda_{\mathrm{reg}}\|x\|_1 + \frac{\beta}{2}\|x\|^2$ and $Q_{i1}(x) = \alpha_i\|D_i x\|_1$ yield the proximal mappings
\[
\begin{aligned}[t]
[\operatorname{prox}_{\gamma P_1}(u)]_j
&= \frac{\operatorname{sign}(u_j)\max\{|u_j| - \gamma\lambda_{\mathrm{reg}}, 0\}}{1 + \gamma\beta}, \\
[\operatorname{prox}_{\gamma Q_{i1}}(u)]_j
&= \operatorname{sign}(u_j)\max\{|u_j| - \gamma\alpha_i w_{ij}, 0\}, \\
[\operatorname{prox}_{H_\lambda}(u)]_j
&= \frac{\operatorname{sign}(u_j)\max\{|u_j| - \lambda_{\mathrm{reg}} - \sum_{i=1}^m \lambda_i \alpha_i w_{ij}, 0\}}{1 + \beta}.
\end{aligned}
\]
Here, $\gamma > 0$ denotes the inner stepsize and $\lambda \in \mathbb{R}_+^m$ represents the multiplier vector. The first two mappings are employed in the inner iterations, whereas $\operatorname{prox}_{H_\lambda}$ is reserved for offline KKT evaluations.
\subsubsection{Numerical results for robust regression}%\label{subsubsec:robust-results}

Both methods successfully attained $R_{\mathrm{KKT}}^{\mathrm{num}} \le 1$ across all $450$ instances. MGSPG required less wall-clock time to reach this criterion in $449$ of the $450$ paired runs (Table~\ref{tab:groupwise}). The median paired hit-time ratio $t_{\mathrm{hit}}^{\mathrm{I}} / t_{\mathrm{hit}}^{\mathrm{M}}$ increased from $\mathbf{1.28}$ at $(n, m) = (60, 30)$ to $\mathbf{3.85}$ at $(240, 480)$. Specifically, at $(n, m) = (240, 480)$, the median hit times were $\mathbf{5.23}\,\mathrm{s}$ for MGSPG versus $\mathbf{19.82}\,\mathrm{s}$ for IDCA--SSG.

\begin{herefigure}
\centering
\includegraphics[width=0.94\textwidth,clip]{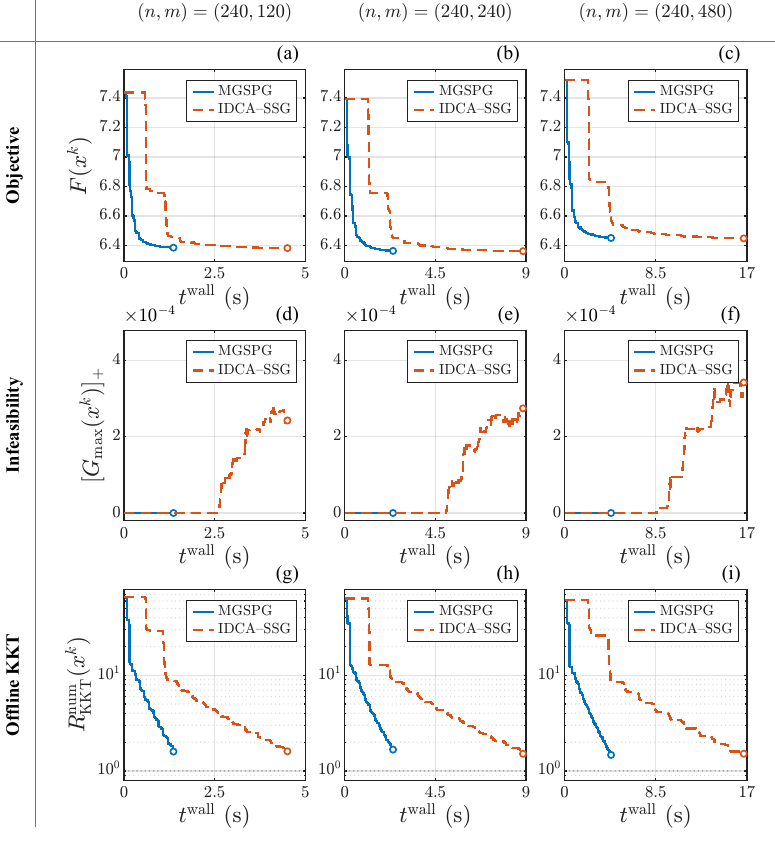}
\captionof{figure}{Performance comparison of MGSPG and IDCA--SSG for groupwise-constrained regression.}
\label{fig:groupwise-traj}
\end{herefigure}

The objective value and numerical KKT residual decreased faster under MGSPG across all three $n = 240$ configurations (Figure~\ref{fig:groupwise-traj}), with all recorded MGSPG outer iterates remaining strictly feasible. Each method's median curve terminates at its earliest hit time over the $50$ instances (marked by circles). For each fixed ratio $m/n \in \{1/2, 1, 2\}$, the median hit time grew more moderately for MGSPG than for IDCA--SSG as $n$ and $m$ increased over the tested scales (Figure~\ref{fig:groupwise-scaling}). As reflected in Table~\ref{tab:groupwise}, MGSPG is faster than IDCA--SSG, with the computational acceleration becoming increasingly pronounced as the problem size scales up.

MGSPG handles the convex terms $P_1$ and $Q_{i1}$ through closed-form proximal mappings, whereas IDCA--SSG uses subgradients of the complete surrogate functions. Under the inner-iteration limits used in these experiments, MGSPG performs fewer outer iterations at a fixed inner-iteration budget and therefore recomputes the group loss gradients less often. This reduces repeated gradient computation, whose cost increases with $n$ and $m$.

\FloatBarrier

\begin{table}[!htbp]
\caption{Median results for groupwise-constrained regression (M: MGSPG, I: IDCA--SSG). Ratio is the median of the paired values $t_{\mathrm{hit}}^{\mathrm{I}}/t_{\mathrm{hit}}^{\mathrm{M}}$. Terminal values are evaluated at each method's termination. MGSPG's constraint-violation medians are zero.}
\label{tab:groupwise}
\centering\fontsize{8}{9.5}\selectfont\setlength{\tabcolsep}{4.5pt}
\begin{tabular}{@{}c rr c rr rr c@{}}
\toprule
& \multicolumn{2}{c}{$t_{\mathrm{hit}}^{\mathrm{wall}}$ (s)} &
& \multicolumn{2}{c}{$F(x^{\mathrm{term}})$}
& \multicolumn{2}{c}{$R_{\mathrm{KKT}}^{\mathrm{num}}(x^{\mathrm{term}})$}
& $[G_{\max}(x^{\mathrm{term}})]_+$\\
\cmidrule(lr){2-3}\cmidrule(lr){5-6}\cmidrule(lr){7-8}\cmidrule(l){9-9}
$(n,m)$ & M & I & Ratio & M & I & M & I & I $(\times10^{-4})$\\
\midrule
(60, 30)   & 0.474 &  0.591 & 1.28 & 7.87078 & 7.87020 & 0.0931 & 0.293 & 4.95\\
(60, 60)   & 0.780 &  1.425 & 1.81 & 7.90429 & 7.90301 & 0.102  & 0.277 & 6.41\\
(60, 120)  & 1.051 &  2.019 & 1.98 & 8.09603 & 8.09583 & 0.0895 & 0.247 & 4.30\\
\addlinespace
(120, 60)  & 0.920 &  1.909 & 2.09 & 6.97914 & 6.97828 & 0.116  & 0.357 & 4.21\\
(120, 120) & 1.418 &  3.163 & 2.30 & 6.85903 & 6.85781 & 0.108  & 0.350 & 4.05\\
(120, 240) & 2.652 &  6.648 & 2.54 & 6.92114 & 6.91939 & 0.106  & 0.308 & 4.05\\
\addlinespace
(240, 120) & 1.746 &  5.634 & 3.28 & 6.37989 & 6.37908 & 0.115  & 0.473 & 2.78\\
(240, 240) & 3.046 & 10.985 & 3.59 & 6.35756 & 6.35681 & 0.117  & 0.476 & 3.63\\
(240, 480) & 5.231 & 19.817 & 3.85 & 6.44365 & 6.44272 & 0.104  & 0.410 & 3.48\\
\bottomrule
\end{tabular}
\end{table}

MGSPG stopped on the displacement criterion in 367 runs and reported insufficient inner accuracy in 83 runs. All 83 returned iterates satisfied $R_{\mathrm{KKT}}^{\mathrm{num}}\le1$. IDCA--SSG used the full budget in all 450 runs.

\begin{herefigure}
\centering
\includegraphics[width=0.94\textwidth,trim={0bp 6bp 0bp 10bp},clip]{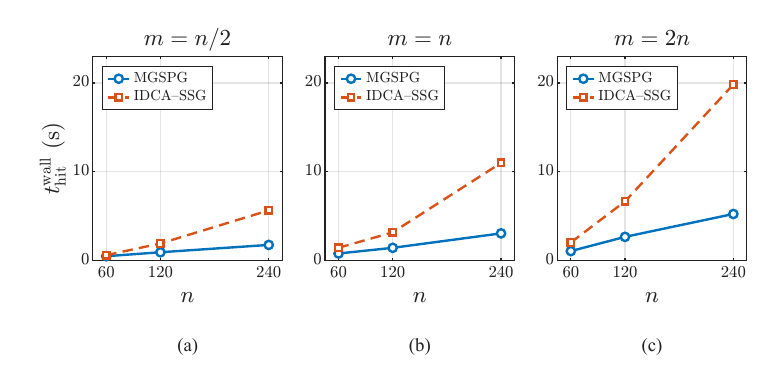}
\captionof{figure}{Scalability comparison of MGSPG and IDCA--SSG for groupwise-constrained regression.}
\label{fig:groupwise-scaling}
\end{herefigure}
\FloatBarrier
%\clearpage

\section{Conclusion}\label{Conclusion}

To exploit the individual proximal structures of composite constraints,
we developed a majorization-guided switching proximal gradient (MGSPG)
method for difference-of-convex optimization with multiple inequality
constraints. The method approximately solves strongly convex surrogate
problems by switching between objective and constraint updates, without
projection onto the joint surrogate feasible set or multiplier computation. 
We established an $O(\tau^{-1})$ inner iteration bound. Under boundedness of the feasible level set and a uniform Slater condition for the surrogates,
MGSPG finds an $\varepsilon$-KKT point in $O(\varepsilon^{-2})$ outer
iterations and $O(\varepsilon^{-4})$ selected component proximal gradient
updates. Numerical experiments demonstrate
the efficacy and competitive performance of the proposed method.

\end{document}